\documentclass[leqno,10pt]{amsart}
\usepackage{amssymb,amsmath,amsfonts,amsthm,mathtools,latexsym,ifthen,bezier}
\usepackage{graphicx}
\usepackage{undertilde}
\usepackage{amsxtra}
\usepackage{url}
\usepackage[english]{babel}

\newcommand{\kla}[1]{ {\langle #1 \rangle} }

\newcommand{\st}{\;|\;}

\newcommand{\const}{ {\rm const} }
\newcommand{\dom}{ {\rm dom} }
\newcommand{\ran}{ {\rm ran} }

\newcommand{\Ult}{{\rm Ult}}

\newcommand{\crit}{ {\rm crit} }

\newcommand{\sub}{\subseteq}

\newfont{\ssi}{cmssi12 at 12pt}

\newcommand{\eins}{ {1{\rm\hspace{-0.5ex}l}} }
\newcommand{\rest}{{\restriction}}

\newcommand{\On}{ {\rm On} }

\newcommand{\leer}{\emptyset}
\newcommand{\ohne}{\setminus}
\newcommand{\lub}{\mathop{ {\rm lub}} }
\newcommand{\id}{ {\rm id} }
\newcommand{\qedd}[1]{\nopagebreak\hspace*{\fill}$ \Box_{#1} $}

\newenvironment{ea*}{\begin{eqnarray*}}{\end{eqnarray*}}
\newcommand{\claim}[2]{
     \begin{enumerate}
       \item[{#1}] {\em #2}
     \end{enumerate}}

\setbox0=\hbox{$\longrightarrow$}
\newcommand{\To}{\longrightarrow}
\newcommand{\emb}[1]{\longrightarrow_{#1} }

\newcommand{\power}{{\mathcal{P}}}

\newcommand{\calM}{\mathcal{M}}

\newcommand{\ba}{{\bar{a}}}

\newcommand{\bA}{{\bar{A}}}

\newcommand{\bj}{{\bar{j}}}
\newcommand{\bx}{{\bar{x}}}
\newcommand{\balpha}{{\bar{\alpha}}}

\newcommand{\bdelta}{{\bar{\delta}}}

\newcommand{\barf}{{\bar{f}}}

\newcommand{\tf}{{\tilde{f}}}

\newcommand{\bp}{{\bar{p}}}

\newcommand{\talpha}{{\tilde{\alpha}}}

\newcommand{\tA}{{\tilde{A}}}

\newcommand{\vf}{{\vec{f}}}

\newcommand{\vx}{{\vec{x}}}

\newcommand{\vA}{{\vec{A}}}

\newcommand{\vM}{{\vec{M}}}

\newcommand{\vU}{{\vec{U}}}

\newcommand{\vkappa}{{\vec{\kappa}}}
\newcommand{\vlambda}{{\vec{\lambda}}}
\newcommand{\veta}{{\vec{\eta}}}

\newcommand{\vpi}{{\vec{\pi}}}

\newcommand{\seq}[2]{{\langle#1\;|\;}\linebreak[0]{#2\rangle}}

\renewcommand{\phi}{\varphi}

\newcommand{\ZFC}{\ensuremath{\mathsf{ZFC}}}

\newcommand{\V}{\ensuremath{\mathrm{V}}}

\newcommand{\TC}{\mathop{\mathsf{TC}}}
\def\<#1>{\langle#1\rangle}

\newcommand{\B}{{\mathord{\mathbb{B}}}}
\renewcommand{\P}{{\mathord{\mathbb P}}}

\newcommand{\MP}{\ensuremath{\mathsf{MP}}}

\newcommand{\ColNothing}{\mathrm{Col}}
\newcommand{\Col}[1]{\ColNothing(#1)}

\newcommand{\MPColNothing}[1]{\MP_{\Col{\dot{\kappa}}}}

\newcommand{\isomorphic}{\cong}

\newcommand{\Add}{\mathord{\mathrm{Add}}}

\newcommand{\isomorphism}{\stackrel{\sim}{\longleftrightarrow}}

\newcommand{\Los}{{\L}{o}{\'s}}

\newcommand{\Prikry}{P\v{r}\'{\i}kr\'{y}}

\newcommand{\prooff}[1]{\noindent{\em Proof of {#1}.~}}

\newtheorem{thm}{Theorem}[section]
\newtheorem*{thm*}{Theorem} 
\newtheorem{cor}[thm]{Corollary}
\newtheorem{lem}[thm]{Lemma}
\newtheorem{obs}[thm]{Observation}

\theoremstyle{definition}
\newtheorem{defn}[thm]{Definition}
\newtheorem{question}[thm]{Question}

\theoremstyle{remark}
\newtheorem{remark}[thm]{Remark}

\usepackage{stmaryrd}

\newcommand{\Bukovsky}{Bukovsk\'{y}}
\newcommand{\MF}{\mathbb{M}}

\newcommand{\mix}{\text{mix}}
\newcommand{\BooleanValue}[1]{\llbracket#1\rrbracket}
\newcommand{\BV}{\BooleanValue}
\newcommand{\vG}{\vec{G}}

\newcounter{AllPurposeCounter}

\begin{document}

\title[Boolean Ultrapowers and the Bukovsk\'{y}-Dehornoy Phenomenon]{Boolean Ultrapowers, the Bukovsk\'{y}-Dehornoy Phenomenon, and Iterated Ultrapowers}

\author[Fuchs]{Gunter Fuchs}
\address[G.~Fuchs]{Mathematics,
          The Graduate Center of The City University of New York,
          365 Fifth Avenue, New York, NY 10016
          \&
          Mathematics,
          College of Staten Island of CUNY,
          Staten Island, NY 10314}
\email{Gunter.Fuchs@csi.cuny.edu}
\urladdr{http://www.math.csi.cuny/edu/$\sim$fuchs}

\author[Hamkins]{Joel David Hamkins}
 \address[J.~D.~Hamkins]{Mathematics, Philosophy, Computer Science,
          The Graduate Center of The City University of New York,
          365 Fifth Avenue, New York, NY 10016
          \&
          Mathematics,
          College of Staten Island of CUNY,
          Staten Island, NY 10314}
\email{jhamkins@gc.cuny.edu}
\urladdr{http://jdh.hamkins.org}

\thanks{The research of the first author has been supported in part by PSC CUNY research grant 68604-00 46.}
\keywords{Boolean ultrapowers, iterated ultrapowers, Prikry forcing, Magidor forcing, large cardinals}
\subjclass[2010]{03E35, 03E40, 03E45, 03E55, 03C20}

\begin{abstract}
We show that while the length $\omega$ iterated ultrapower by a normal ultrafilter is a Boolean ultrapower by the Boolean algebra of \Prikry{} forcing, it is consistent that no iteration of length greater than $\omega$ (of the same ultrafilter and its images) is a Boolean ultrapower. For longer iterations, where different ultrafilters are used, this is possible, though, and we give Magidor forcing and a generalization of \Prikry{} forcing as examples.
We refer to the discovery that the intersection of the finite iterates of the universe by a normal measure is the same as the generic extension of the direct limit model by the critical sequence as the \Bukovsky-Dehornoy phenomenon, and we develop a sufficient criterion (the existence of a simple skeleton) for when a version of this phenomenon holds in the context of Boolean ultrapowers. Assuming that the canonical generic filter over the Boolean ultrapower model has what we call a continuous representation, we show that the Boolean model consists precisely of those members of the intersection model that have continuously and eventually uniformly represented codes.
\end{abstract}

\maketitle

\section{Introduction}

There are two conspicuous features of \Prikry{} forcing (\cite{Prikry:Dissertation}) that we want to elucidate in this paper. Let $\mu$ be a normal measure on the measurable cardinal $\kappa$, and let $\P=\P_\mu$ be \Prikry{} forcing with respect to that measure. Then
\begin{enumerate}
  \item if one iterates the measure $\mu$ through the natural numbers, then the critical sequence is generic over the limit model for that model's version of the \Prikry{} forcing,
  \item the intersection of the finite iterates is the generic extension of the limit model by the critical sequence.
\end{enumerate}
We analyze these features through the lens of Boolean ultrapowers.
Our initial observation regarding 1.~was the rediscovery of a fact that was already observed by \Bukovsky{} in \cite{Bukovsky1977:IteratedUltrapowersAndPrikryForcing}, namely that the limit model can be realized as the Boolean ultrapower of the universe by the ultrafilter on the Boolean algebra of \Prikry{} forcing which results from pulling back the generic filter over the limit model that is generated by the critical sequence. It is very natural to think that this should be the case, because a Boolean ultrapower always comes with a filter that is generic over the model produced. Letting $\B$ be the Boolean algebra and $U$ the ultrafilter on it, we write $j:\V\To\check{\V}_U$ for the Boolean ultrapower and the elementary embedding. We will mostly be interested in the case where it is well-founded, and in that case, we take $\check{\V}_U$ to be transitive. The model $\check{\V}_U$ sits inside the model $\V^\B/U$, the full Boolean model. So in our notation, $\V^\B/U$ consists of the equivalence classes $[\sigma]_U$ of $\B$-names $\sigma$, with respect to the equivalence relation $\sim$ defined by letting $\sigma\sim\tau$ iff the Boolean value $\BV{\sigma=\tau}\in U$. The structure $\V^\B/U$ is equipped with a pseudo epsilon relation $E$, where $[\sigma]_UE[\tau]_U$ iff $\BV{\sigma\in\tau}\in U$. Again, in the case we are mostly interested in, $E$ is well-founded and extensional, so we can take $(\V^\B/U,E)$ to be transitive, and $E$ becomes the $\in$ relation. There is a special element in $\V^\B/U$, namely $G=[\dot{G}]_U$, where $\dot{G}$ is the canonical name for the generic filter. The filter $G$ is generic over the inner model of $\V^\B/U$ which consists only of the equivalence classes of those names $\sigma$ with $\BV{\sigma\in\check{\V}}\in U$. This is the Boolean ultrapower of $\V$ by $U$, and we write $\check{\V}_U$ for this model. There is an obvious elementary embedding from $\V$ to $\check{\V}_U$, the Boolean ultrapower map defined by $j(x)=[\check{x}]_U$. Then $G$ is $j(\B)$-generic over $\check{\V}_U$, and $\check{\V}_U[G]=\V^\B/U$. An in-depth exploration of this construction is undertaken in \cite{HamkinsSeabold:BULC}, and we will refer to this paper frequently and use the terminology used there.

There are other forcing notions that are ``accompanied'' by an appropriate iteration, in the sense that the critical sequence is generic over the limit model for the limit model's version of the forcing itself, and it is natural to expect the same situation to arise again: the limit model can be realized as a single Boolean ultrapower, and the forcing extension is the corresponding Boolean model. We show in Section \ref{section:IteratedUltrapowersAsSingleBooleanUltrapowers} that this occurs in the case of \Prikry{} forcing, Magidor forcing and a generalization of \Prikry{} forcing. We also show that it is consistent that no longer iterations of one measure can be realized as a single Boolean ultrapower.

Regarding 2., there is an enticing scenario for a connection to Boolean ultrapowers. The Boolean ultrapower can be regarded as a direct limit of models $M_A$, indexed by maximal antichains in $\B$. The model $M_A$ is the ultrapower of $\V$ by the ultrafilter $U_A$ on $A$ which consists of those subsets $X$ of $A$ whose join is in $U$. If $B$ refines $A$, then there is an embedding $\pi_{A,B}:M_A\To M_B$, and the maximal antichains are directed under refinement. The Boolean ultrapower $\check{\V}_U$ is the direct limit of these models and embeddings. So the Boolean ultrapower naturally is situated inside a forcing extension of itself, and it comes with a directed system of models whose limit it is. The question is whether the phenomenon in 2., which we call the \Bukovsky-Dehornoy phenomenon, holds in greater generality. That is, under which circumstances is the Boolean model $\V^\B/U$ (which is the same as the forcing extension $\check{\V}_U[G]$) equal to the intersection of the models $M_A$? We investigate this in Section \ref{section:Bukovsky-Dehornoy}. We develop a sufficient criterion there: the existence of a simple skeleton. We also develop a version of this phenomenon that applies to ill-founded Boolean ultrapowers. A simple skeleton consists of a directed collection $\mathfrak{A}$ of maximal antichains of $\B$ that generates the limit model, such that each $A\in\mathfrak{A}$ is simple, meaning that $\pi_{A,\infty}^{-1}``G\in M_A$, and satisfies the smallness assumption that $j``\mathfrak{A}\in\V^\B/U$. In Section \ref{sec:CEU}, we develop the concept of continuous, eventually uniform representations, resulting, among other things, in a sufficient criterion for when the \Bukovsky-Dehornoy phenomenon holds, without a smallness assumption, but with a slightly strengthened simplicity assumption.

Here is an overview of notations and facts around the Boolean ultrapower construction we use. For more information, we refer to \cite{HamkinsSeabold:BULC}.

\begin{enumerate}
\item $\V^\B/U$ is the model of equivalence classes of $\B$-names, according to the equivalence relation which identifies two $\B$-names $\sigma$ and $\tau$ if the Boolean value $\BooleanValue{\sigma=\tau}$ belongs to $U$. This model is equipped with a pseudo-epsilon relation $E$, according to which $[\sigma]_UE[\tau]_U$ holds iff $\BooleanValue{\sigma\in\tau}\in U$.
\item $\check{\V}_U$ is the Boolean ultrapower of $\V$ by $U$. It can be viewed as a submodel of $\V^\B/U$, and consists only of the equivalence classes of names $\sigma$ such that $\BV{\sigma\in\check{\V}}\in U$. Note that the class of such names is much richer than just names of the form $\check{x}$, since it includes mixtures of such names, and if $U$ is not generic, then these mixtures will not be equivalent to any $\check{x}$. The model consisting only of equivalence classes of names of the latter form would clearly be isomorphic to $\V$ itself!
\item There is a canonical embedding $j:\V\To\check{\V}_U$, defined by $j(x)=[\check{x}]_U$. To indicate that $j$ is that particular embedding, we may write $j:\V\emb{U}\check{\V}_U$.
\item If $A\sub\B$ is a maximal antichain, then $U_A$ is the ultrafilter on $A$ consisting of subsets $X$ of $A$ such that $\bigvee X\in U$.
\item $M_A=\V^A/U_A$ is the ultrapower of $\V$ by $U_A$, and $\pi_{0,A}:\V\To_{U_A}\V_A$ is the ultrapower embedding. Even though $U$ is not mentioned in the notation $M_A$, it will be clear from the context which ultrafilter is used.
\setcounter{AllPurposeCounter}{\value{enumi}}
\end{enumerate}
As usual, if $(\V^\B/U,E)$ is well-founded, which we assume below, then we take it to be transitive.
\begin{enumerate}
\setcounter{enumi}{\value{AllPurposeCounter}}
  \item If $A\in\check{\V}_U$ is a maximal antichain in $j(\B)$, then $A$ intersects $G=[\dot{G}]_U$ in exactly one condition which we denote by $b_A$. So $b_A$ is defined by $A\cap G=\{b_A\}$.
  \item $\check{\V}_U$ is the direct limit of all $M_A$, where $A$ is a maximal antichain in $\B$. These maximal antichains are ordered by refinement: $B\le^*A$ if $B$ refines $A$, meaning that for every $b\in B$, there is an $a\in A$ (which is unique) such that $b\le a$. This partial ordering makes the collection of maximal antichains of $\B$ a (downward) directed partial order, and there are canonical embeddings $\pi_{A,B}:M_A\To M_B$ which make this collection of ultrapowers and embeddings a directed system. Denote the direct limit embeddings by \[\pi_{A,\infty}:M_A\To\check{\V}_U.\]
\end{enumerate}

Otherwise, our notation should be standard. It may be worth pointing out that in the context of a partial order $\kla{\P,\le}$, for $p,q\in\P$, we write $p||q$ to express that $p$ and $q$ are compatible, meaning that there is an $r\in\P$ with $r\le p$ and $r\le q$. If the partial order is a Boolean algebra $\B$, then we often tacitly work with $\B\ohne\{0\}$. So in that case, $p||q$ would mean $p\land q\neq 0$, and an antichain in $\B$ is really an antichain in $\B\ohne\{0\}$. Similarly, $p\perp q$ means that $p$ and $q$ are incompatible. In the context of forcing, the symbol $||$ is used with a different meaning: if $\phi$ is a formula in the forcing language of $\P$, and $p$ is a condition in $\P$, then $p||\phi$ means that $p$ decides $\phi$, that is, either $p$ forces $\phi$, or $p$ forces $\neg\phi$.

The paper is organized as follows. In Section \ref{section:IteratedUltrapowersAsSingleBooleanUltrapowers}, we give three examples of canonical iterated ultrapowers that can be presented as single Boolean ultrapowers. The Boolean algebras here are those associated to \Prikry{} forcing, Magidor forcing and generalized \Prikry{} forcing. We also give a lower bound for the consistency strength of the assumption that an iterated ultrapower by one normal ultrafilter and its images of length greater than $\omega$ can be presented as a single Boolean ultrapower. In Section \ref{section:Bukovsky-Dehornoy}, we develop a sufficient criterion for when the \Bukovsky-Dehornoy phenomenon holds, that is, that the intersection model is equal to the Boolean model: the existence of a simple skeleton. Finally, in Section \ref{sec:CEU}, we develop the theory of continuous, eventually uniform representations, in order to describe the part of the intersection model that makes up the Boolean model, if the canonical generic filter $[\dot{G}]_U$ is uniformly represented.

\section{Iterated ultrapowers as single Boolean ultrapowers}
\label{section:IteratedUltrapowersAsSingleBooleanUltrapowers}

In this section, we will show three instances where iterated ultrapowers can be described as single Boolean ultrapowers. These forcing notions/Boolean algebras used will be used throughout the remainder of the paper.

\subsection{\Prikry{} forcing and the $\omega$-th iterate of a normal measure}
\label{subsec:PrikryForcingAndOmegathIterate}

Let $\mu$ be a normal measure on $\kappa$, and let $(M_n\st n<\omega)$ be the iterates of $\V$ by that measure, and let $\pi_{m,n}:M_m\To M_n$ be the canonical embeddings. Let's also set $\mu_n=\pi_{0,n}(\mu)$ and $\kappa_{n}=j_{0,n}(\kappa)$. So we have $M_0=\V$, $\kappa_0=\kappa$, $\mu_0=\mu$, and for every $n<\omega$, $\pi_{n,n+1}:M_n\To_{\mu_n}M_{n+1}$ is the ultrapower embedding from $M_n$ into its ultrapower $M_{n+1}$ by $\mu_n$, which is a normal measure in the sense of $M_{n+1}$ on $\kappa_{n}$. Let $M_\omega$ be the direct limit of this system of embeddings, with direct limit embeddings $\pi_{n,\omega}:M_n\To M_\omega$, and let $\kappa_\omega=j_{0,\omega}(\kappa_0)$ (so $\kappa_\omega=\sup_{n<\omega}\kappa_n$), and let $\mu_\omega=\pi_{0,\omega}(\mu)$.

There are well-known yet still fascinating connections between \Prikry{} forcing and this system of embeddings. Let's write $\P_\mu$ for \Prikry{} forcing with respect to the fixed normal measure $\mu$.
\begin{enumerate}
  \item The sequence $\vkappa=\seq{\kappa_n}{n<\omega}$ is $\pi_{0,\omega}(\P_\mu)$-generic over $M_\omega$.
  \item The model $\bigcap_{n<\omega}M_n$ is the forcing extension of $M_\omega$ by that sequence: $\bigcap_{n<\omega}M_n=M_\omega[\vkappa]$.
\end{enumerate}
The first fact is due to Solovay, relying on a characterization of \Prikry-generic sequences due to Mathias \cite{Mathias1973:PrikryGenericSequences}. The formulation is a little sloppy, and since we need to be more precise later, let's clarify right now that the sequence $\vkappa$ gives rise to a filter $G_\vkappa$ on $j_{0,\omega}(\P_\mu)$ that's defined by
\[G_\vkappa=\{\kla{\vkappa\rest n,A}\st A\in\mu_\omega,\ \text{for}\ i<n,\ \kappa_i<\min(A),\ \text{and for all}\ i\in[n,\omega),\ \kappa_i\in A\},\]
and it is this filter that is claimed to be generic in 1.~above.

The second connection is due to \Bukovsky{} and Dehornoy, independently, and we shall explore in the next section how it can be generalized to the realm of Boolean ultrapowers.

There is another fascinating and less well-known connection that was originally observed by Bukovsky in \cite{Bukovsky1977:IteratedUltrapowersAndPrikryForcing}, and that we rediscovered. Let us denote the Boolean algebra of the \Prikry{} forcing $\P_\mu$ by $\B_\mu$, and let $G^*_\vkappa$ be the ultrafilter in $j_{0,\omega}(\B_\mu)$ generated by $G_\vkappa$. Let $U=j^{-1}``G^*_\vkappa$, which we shall call the canonical ultrafilter on $\B_\mu$ (it is obvious that it is indeed an ultrafilter). We will give a short alternative proof here, which relies on the following general fact about Boolean ultrapowers.

\begin{thm}[{\cite[Theorem 38]{HamkinsSeabold:BULC}}]
\label{thm:HamkinsSeaboldCriterionForBeingBooleanUltrapower}
Suppose that $\B$ is a complete Boolean algebra and $j:\V\To\bar{\V}$ is an embedding such that there is a filter $F\sub j(\B)$ that is $\ran(j)$-generic and
\[\bar{\V}=\{j(f)(b_A)\st A\sub\B\ \text{is a maximal antichain and}\ f:A\To\V\},\]
where $b_A$ is the unique member of $F\cap j(A)$. Then $j:\V\To\bar{\V}$ is isomorphic to the Boolean ultrapower of $\V$ by $U:=j^{-1}``F\sub\B$. In this case, $F$ is actually $\bar{\V}$-generic, and $\bar{\V}[F]$ is isomorphic to $\V^\B/U$.
\end{thm}

\begin{thm}[\Bukovsky, \cite{Bukovsky1977:IteratedUltrapowersAndPrikryForcing}]
\label{thm:OmegathIteratedUPIsBooleanUltrapowerOfPrikryForcing}
The $\omega$-th iterate $M_\omega$ of $\V$ by the normal measure $\mu$ on $\kappa$ is the Boolean ultrapower of $\V$ by the canonical ultrafilter $U$ on the Boolean algebra $\B:=\B_\mu$ of \Prikry{} forcing with respect to $\mu$, and the iterated ultrapower embedding $j_{0,\omega}$ is the same as the Boolean ultrapower embedding $j_U$.
\end{thm}

\begin{proof} In order to apply the previous theorem with $j=\pi_{0,\omega}$, $\bar{\V}=M_\omega$ and $F=G^*_\vkappa$, it needs only to be checked that every member of $M_\omega$ has the form $\pi_{0,\omega}(f)(b_A)$, for some maximal antichain $A\sub\B$ and $f:A\To\V$, where $b_A$ is defined as in the theorem we are about to apply.

To see this, let $x\in M_\omega$ be given. Let $x=\pi_{0,\omega}(f)(\kappa_0,\ldots,\kappa_{n-1})$, where $f\in\V$ is a function whose domain is the set of increasing length $n$ sequences of ordinals less than $\kappa$. For notational simplicity, construe $\P_\mu$ as a dense subset of $\B$. If $s$ is a finite increasing sequence of ordinals less than $\kappa$, then let $s^*=\kla{s,\kappa\ohne\lub(\ran(s))}\in\P_\mu$ (so $s^*$ is the weakest condition with first coordinate $s$). Let
\[ A_n=\{s^*\st s:n\To\kappa\ \text{is strictly increasing}\}.\]
Obviously, $A_n$ is an antichain. It is also maximal in $\P$, and hence in $\B$, since $\P$ is dense in $\B$: Let $\kla{t,S}\in\P_\mu$ be given. If $|t|\ge n$, then $\kla{t,S}\le\kla{t\rest n,\kappa\ohne\lub(\ran(t))}$. If $|t|<n$, then $t$ can be extended to an increasing sequence $t'$ of length $n$ by appending the next $n-|t|$ elements of $S$ to $t$. Clearly then, $\kla{t',S\ohne\lub(\ran(t'))}$ is a common extension of $\kla{t,S}$ and $(t')^*$. So in both cases, we found a condition in $A_n$ compatible with $\kla{t,S}$.

It is obvious now that $b_{A_n}=\kla{\kappa_0,\ldots,\kappa_{n-1}}^*$, where we ambiguously use the notation $t^*$ in the context of $j_{0,\omega}(\P_\mu)$ with the obvious meaning. So if we define $f^*:A_n\To\V$ by $f^*(s^*)=f(s)$, it follows that \[x=\pi_{0,\omega}(f)(\kla{\kappa_0,\ldots,\kappa_{n-1}})=\pi_{0,\omega}(f^*)(\kla{\kappa_0,\ldots,\kappa_{n-1}}^*)=\pi_{0,\omega}(f^*)(b_{A_n}).\]
This shows that the theorem quoted above can be applied, completing the proof.
\end{proof}

\subsection{Longer iterations}
\label{subsec:LongerIterations}

In this section, we will see two examples of iterations of transfinite length that can be realized as single Boolean ultrapowers. First, let us point out some limitations, though.

\subsubsection{A limitation}

In the following theorem, we use the usual notation for the Mitchell order on normal ultrafilters. If $\mu$ and $\nu$ are normal ultrafilters on $\kappa$, then $\mu$ is less than $\nu$ in the Mitchell order if $\mu$ belongs to the transitivized ultrapower of $\V$ by $\nu$. It is well-known that this order is well-founded. If $\rho$ is the rank function on the set of normal ultrafilters on $\kappa$, then the range of $\rho$ is the order of $\kappa$, denoted $o(\kappa)$, and the order of a normal ultrafilter $\mu$ on $\kappa$ is $\rho(\mu)$, denoted $o(\mu)$.

\begin{thm}
If there is a measurable cardinal $\kappa$ and an $\alpha>\omega$ such that the $\alpha$-th iterate of $\V$ by a normal ultrafilter $\mu$ on $\kappa$ is a Boolean ultrapower of $\V$, then there is an inner model $N$ with a measurable cardinal $\kappa'$ such that $o(\kappa')^N>\kappa'$.
\end{thm}

\begin{proof} Note that we are only assuming that the $\alpha$-th iterate of $\V$ by $\mu$, as a model, is a Boolean ultrapower, and not necessarily that the embedding from $\V$ into that iterate is a Boolean ultrapower embedding.

Let's assume that there is no inner model $N$ with a measurable cardinal $\kappa'$ such that in $N$, $o(\kappa')=(\kappa')^{++}$, because otherwise, we're done. Let $K$ be the core model. Then by a result of Mitchell \cite{Mitchell:CoreModelForSequencesOfMeasuresI}, every countably complete $K$-ultrafilter is in $K$ (where countable completeness here means that the intersection of countably many measure 1 sets is nonempty).

Now let $\mu$ be a normal ultrafilter on $\kappa$, and let $\seq{M_i}{i\le\alpha}$ be the iteration of $\V$ by $\mu$, with embeddings $\seq{\pi_{i,j}}{i\le j\le\alpha}$. Let $\mu_i=\pi_{0,i}(\mu)$, so that $\pi_{i,i+1}:M_i\emb{\mu_i}M_{i+1}$. Also, let $\B$ be a complete Boolean algebra, $U$ an ultrafilter on $\B$, and let $j:\V\emb{U}M_\alpha$ be the Boolean ultrapower, so $M_{\alpha}=\check{\V}_U$.
Let $G\sub j_U(\B)$ be the canonical $\check{\V}_U$-generic filter, $G=[\dot{G}]_U$. By \cite[Theorem 28]{HamkinsSeabold:BULC}, $\check{V}_U[G]$ is $\lambda$-closed in $\V$ iff $j``\lambda\in\check{V}_U$. In particular, $M_\alpha[G]=\check{\V}_U[G]$ is $\omega$-closed in $\V$ (and much more, of course, but this is all we need). Let $\kla{\kappa_\gamma\st\gamma<\alpha}$ be the sequence of critical points of the iteration, and let $C=\{\kappa_n\st n<\omega\}$. Then $C\in M_\alpha[G]$, by the closure. Clearly, for any $A\sub\kappa_\omega$ with $A\in M_\omega$, we have that $A\in\mu_\omega$ iff $C\ohne A$ is finite. Since $\power(\kappa_\omega)\cap M_\omega=\power(\kappa_\omega)\cap M_\alpha$, it follows that $\mu_\omega$ is definable in $M_\alpha[G]$, using $\power(\kappa_\omega)^{M_\alpha}$ as a parameter, and so, $\mu_\omega\in M_\alpha[G]$. So, since $\mu_\omega\sub M_\alpha$, this means that there is a minimal model of $\ZFC$, $M_\alpha[\mu_\omega]$, that contains $M_\alpha\cup\{\mu_\omega\}$, and this model is a forcing extension of $M_\alpha$. But since $M_\alpha\cup\{\mu_\omega\}\sub M_\omega$ (as $M_\alpha$ is definable in $M_\omega$), it follows that $M_\alpha[\mu_\omega]\sub M_\omega$. So since $\mu_\omega$ is a normal ultrafilter in $M_\omega$, it follows that $\mu_\omega$ is also a normal ultrafilter in $M_\alpha[\mu_\omega]$.

Using the result of Mitchell quoted above in $M_\alpha[\mu_\omega]$, since $\mu_\omega\cap K^{M_\alpha[\mu_\omega]}$ clearly is a countably closed $K^{M_\alpha[\mu_\omega]}$-ultrafilter, we can conclude that this filter is in $K^{M_\alpha[\mu_\omega]}$. By the forcing-absoluteness of the core model, $K^{M_\alpha}=K^{M_\alpha[\mu_\omega]}$. So, $\mu_\omega\cap K^{M_\alpha}\in K^{M_\alpha}$. Further, if $\alpha>\omega+1$, then $\crit(\pi_{\omega+1,\alpha})=\kappa_{\omega+1}$ and since $\power(\kappa_{\omega+1})^{M_{\omega+1}}=\power(\kappa_{\omega+1})^{M_{\alpha}}$, it follows that $\mu_\omega\cap K^{M_{\alpha+1}}=\mu_\omega\cap K^{M_{\omega+1}}\in K^{M_{\omega+1}}$. Pulling this back via $\pi_{0,\omega}^{-1}$ results in
\[\mu\cap K^{M_1}\in K^{M_1}\]
Let us now forget about the iteration and write $\pi=\pi_{0,1}$ and $M=M_1$. So we have
\[\pi:\V\emb{\mu}M\ \text{and}\ \mu\cap K^M\in K^M\]
For $\delta<\kappa$, letting
\[A_\delta=\{\gamma<\kappa\st o(\gamma)^K\ge\delta\}\]
we show by induction that $A_\delta\in\mu$.

The case $\delta=0$ is trivial, and the limit case is also clear, since in that case, $A_\delta=\bigcap_{\bdelta<\delta}A_{\bdelta}$.

So let's assume that $A_{\delta-1}\in\mu$. Then
\[A_{\delta-1}=\pi(A_{\delta-1})\cap\kappa=\{\gamma<\kappa\st o(\gamma)^{K^M}\ge\delta-1\}\in\mu\cap K^M\]
so that $o(\mu\cap K^M)^{K^M}\ge\delta-1$, which means that $o(\kappa)^{K^M}\ge\delta$. But then,
\[\kappa\in\pi(A_\delta)\]
which means that $A_\delta\in\mu$, by normality of $\mu$.

Applying Mitchell's result in $\V$ yields, of course, that $\mu\cap K\in K$, and clearly, $A_\delta\in\mu\cap K$, for all $\delta<\kappa$. This means that $o(\mu\cap K)^K\ge\kappa$, and hence, $o^K(\kappa)>\kappa$. \end{proof}

\medskip

In order to formulate another similar restriction, let's say that a cardinal $\kappa$ is \emph{absolutely the least measurable cardinal} if $\kappa$ is measurable and there is no poset $\P$ that forces that there is a smaller measurable cardinal. Note that every model with a measurable cardinal has a forcing extension with an absolutely least measurable cardinal. It is understood that when writing ``forcing'', we refer to ``set forcing'', unless we specifically write ``class forcing.''

\begin{thm}
If $\kappa$ is absolutely the least measurable cardinal, then for any $\alpha>\omega$, the $\alpha$-th iterate of $\V$ by a normal ultrafilter on $\kappa$ is not a Boolean ultrapower of $\V$.
\end{thm}

\begin{proof} Using the notation of the proof of the previous theorem, and using the same reasoning, it follows that $\mu^\omega$ is  a normal ultrafilter in $M_\alpha[\mu^\omega]$, a forcing extension of $M_\alpha$. So $\kappa_\omega$ is measurable in $M_\alpha[\mu^\omega]$. As $\kappa$ is absolutely the least measurable cardinal in $\V$, and since this is a first order property, it follows that $\kappa_\alpha$ is absolutely the least measurable cardinal in $M_\alpha$. But $M_\alpha[\mu^\omega]$ is a forcing extension of $M_\alpha$ in which $\kappa_\omega$, a cardinal less than $\kappa_\alpha$, is measurable. This is a contradiction. \end{proof}

\subsubsection{Magidor Forcing}

On the positive side, Magidor forcing allows for examples of long iterations in which the critical point of the next normal ultrafilter applied is the image of the previous one: Let $\seq{U_\gamma}{\gamma<\alpha}$ be a sequence of normal ultrafilters on $\kappa>\alpha$, increasing in the Mitchell order, and let $\seq{M_\gamma}{\gamma\le\alpha}$ and $\seq{j_{\gamma,\delta}}{\gamma\le\delta\le\alpha}$ be defined by letting $M_0=\V$, $\pi_{0,0}=\id$. If $\seq{M_\gamma}{\gamma\le\xi}$ and $\seq{\pi_{\gamma,\delta}}{\gamma\le\delta\le\xi}$ are already defined and $\xi<\alpha$, then let $\pi_{\xi,\xi+1}:M_\xi\To_{j_{0,\xi}(U_\xi)}M_{\xi+1}$, $\pi_{\xi+1,\xi+1}=\id$ and for $\zeta<\xi$, let $\pi_{\zeta,\xi+1}=\pi_{\xi,\xi+1}\circ\pi_{\zeta,\xi}$. If
If $\seq{M_\gamma}{\gamma<\lambda}$ and $\seq{\pi_{\gamma,\delta}}{\gamma\le\delta<\lambda}$ are already defined, where $\lambda\le\alpha$ is a limit ordinal, then $M_\lambda$, $\seq{\pi_{\gamma,\lambda}}{\gamma\le\lambda}$ is the direct limit of that system.
Let $\seq{\kappa_\gamma}{\gamma<\alpha}$ be the sequence of critical points of that iteration. For $\gamma<\delta<\alpha$, let $U_\gamma=[f^\delta_\gamma]_{U_\delta}$ (such functions exist because the $\vU$ sequence is increasing in the Mitchell order). Let $\MF=\MF(\vU,\vf)$ be the Magidor forcing associated to $\vU$ and $\vf$.

\begin{thm}[\cite{Fuchs:MagidorForcing}, \cite{Dehornoy:IteratedUltrapowersChangingCofinalities}]
The sequence $\seq{\kappa_\gamma}{\gamma<\alpha}$ is a Magidor sequence over $M_\alpha$, that is, it gives rise to a filter $G_\vkappa$ which is generic over $M_\alpha$ for the forcing $\pi_{0,\alpha}(\MF)$.
\end{thm}

Let $\MF^*$ be the complete Boolean algebra of $\MF$, and let $G_\vkappa^*$ be the ultrafilter in $\pi_{0,\alpha}(\MF^*)$ generated by $G_\vkappa$.

\begin{thm}
\label{thm:DirectLimitAsUltrapowerOfMagidorAlgebra}
$M_\alpha$ is the Boolean ultrapower of $\V$ by the ultrafilter $F=(\pi_{0,\alpha}^{-1})``G^*_\vkappa$ in the Boolean algebra $\MF^*$.
\end{thm}

\begin{proof} Let $j=\pi_{0,\alpha}:\V\To M_\alpha$. We again want to use Theorem \ref{thm:HamkinsSeaboldCriterionForBeingBooleanUltrapower} (that is, \cite[Theorem 38]{HamkinsSeabold:BULC}). As in the proof of the corresponding theorem for \Prikry{} forcing, let's construe $\MF^*$ to have $\MF$ as a dense subset. For every finite subset $a$ of $\alpha$, let $I_a$ be the collection of $s$ with domain $a$ such that there is a $T$ with $\kla{s,T}\in\MF$. Note that if $s\in I_a$, then there is a weakest $T$ witnessing this, call it $T_s$ (for $\xi<\min(a)$, $T_s(\xi)=s(\min(a))$, for $\xi>\max(a)$, $T_s(\xi)=\kappa$, and for $\min(a)<\xi<\max(a)$, $\xi\notin a$, $T_s(\xi)=s(\min(a\ohne(\xi+1)))\ohne s(\max(a\cap\xi))+1$).
Let
\[A_a=\{\kla{s,T_s}\st s\in I_a\}.\]
It is now easy to see that $A_a$ is a maximal antichain in $\MF$. Clearly, it is an antichain. To see that it is maximal, we have to show that any condition $\kla{r,S}$ in $\MF$ is compatible with some condition in $A_a$. First, it is clear that there is a condition $\kla{r',S'}\le\kla{r,S}$ with $a\sub\dom(r')$ - this follows from  \cite[Lemma 3.2]{Magidor78:ChangingCofinalitiesOfCardinals}. But clearly then, $s:=r'\rest a\in I_a$, and so, $\kla{r',S'}\le\kla{s,T_s}$ and $\kla{r',S'}\le\kla{r,S}$. So $\kla{r,S}$ is compatible with $\kla{s,T_s}\in A_a$. So $A_a$ is a maximal antichain in $\MF$, and hence in $\MF^*$, as $\MF$ is dense in $\MF^*$. To see that the assumptions of Theorem \ref{thm:HamkinsSeaboldCriterionForBeingBooleanUltrapower} are satisfied, let $x\in M_\alpha$. Let $x=j(f)(\kappa_{\xi_0},\ldots,\kappa_{\xi_{n-1}})$, $f:\kappa^n\To\V$. Let $a=\{\xi_0,\ldots,\xi_{n-1}\}$. Define $f^*:A_a\To\V$ by $f^*(\kla{s,T_s})=f(s)$. Clearly, $b_{A_a}$, the unique condition in $j(A_a)$ whose first component has domain $a$ and which belongs to the generic filter corresponding to the critical sequence, must be $\kla{\vkappa\rest a,T_{\vkappa\rest a}}$. So $j(f^*)(b_{A_a})=j(f)(\vkappa\rest a)=x$. The claim now follows from Theorem \ref{thm:HamkinsSeaboldCriterionForBeingBooleanUltrapower}. \end{proof}

\subsubsection{Generalized \Prikry{} Forcing}
\label{subsubsec:GeneralizedPrikryForcing}

Another example where long iterations can be realized as single Boolean ultrapowers is a generalization of \Prikry{} forcing that was analyzed in great detail in  \cite{Fuchs:COPS}. The starting point for such forcing notions is a discrete set of measurable cardinals $D$, with monotone enumeration $\vec{\kappa}=\seq{\kappa_i}{i<\alpha}$ and order type $\alpha$, a corresponding sequence $\vec{U}=\seq{U_i}{i<\alpha}$ such that $U_i$ is a normal ultrafilter on $\kappa_i$, and a sequence $\seq{\eta_i}{i<\alpha}$ of ordinals in $[1,\omega]$. The forcing $\P=\P_{\vU,\veta}$ will add a set of ordinals of order type $\eta_i$ below $\kappa_i$, for each $i<\alpha$. In case $\eta_i=\omega$, that set will be cofinal in $\kappa_i$, so that the cofinality of $\kappa_i$ will become $\omega$. If $\eta_i<\omega$, then the cofinality of $\kappa_i$ will remain unchanged. Conditions in $\P$ are pairs $\kla{s,T}$, where $s$ is a function whose domain is a finite subset of $\alpha$, and for every $i\in\dom(s)$, $s(i)\sub\kappa_i\ohne\sup_{j<i}\kappa_j$ is finite and has size in $[1,\eta_i]$. By convention, $s(i)$ is taken to be $\leer$ if $i\notin\dom(s)$, and similarly for $T(i)$. The domain of the function $T$ consists of all $i<\alpha$ with $|s(i)|<1+\eta_i$, and $T(i)\in U_i$, $s(i)\sub\min(T(i))$ for all $i\in\dom(T)$. The ordering is defined in the natural way: $\kla{s',T'}\le\kla{s,T}$ if for all $i<\alpha$, $s(i)\sub s'(i)$, $s'(i)\ohne s(i)\sub T(i)$, and for all $i<\alpha$, $T'(i)\sub T(i)$.

Here, we will focus on the case that $\eta_i=1$, for all $i<\alpha$, to simplify the notation - everything should go through in the general setting as well. So we will suppress any mention of $\veta$.

In \cite{Fuchs:COPS}, an iterated ultrapower, called the imitation of $\P$, was constructed. In the special case where $\eta_i=1$, the construction proceeds as follows. We start with the model $M_0=\V$, $\pi_{0,0}=\id$. If $M_i$ has been constructed already, together with embeddings $\pi_{\gamma,\delta}$, for $\gamma\le\delta\le i$, then we let
\[\pi_{i,i+1}:M_i\emb{W_i}M_{i+1}\]
be the ultrapower embedding by $W_i=\pi_{0,i}(\vU)_i$, where $M_{i+1}$ is transitive. Then, as usual, for $\gamma<i$, we let $\pi_{\gamma,i+1}=\pi_{i,i+1}\pi_{\gamma,i}$. At limit $\lambda$, we let $M_\lambda$, together with the embeddings $\pi_{\gamma,\lambda}$, for $\gamma<\lambda$, be the transitivized direct limit of the system constructed thus far. We carry out this construction until a point is reached at which $W_i$ is undefined, that is, a point at which $i=\pi_{0,i}(\alpha)$. Let us write $\tilde{\alpha}$ for that point, so $M_{\tilde{\alpha}}$ is the last model in the imitation iteration of $\P$. Let us also write $\lambda_i=\pi_{0,i}(\vkappa)_i$. So this is the critical point of $\pi_{i,i+1}$; it is the measurable cardinal (in the sense of $M_i$) that $W_i$ lives on. It will also be a useful shorthand to write $\alpha_i$ for $\pi_{0,i}(\alpha)$. So $\tilde{\alpha}=\alpha_{\tilde{\alpha}}$, and $\tilde{\alpha}$ is least with that property. The main fact on the imitation iteration of $\P$ that we need here is the following corollary (adapted to the special case $\eta_i=1$):

\begin{thm}[{\cite[Corollary 2]{Fuchs:COPS}}]
\label{thm:CriticalSequenceOfImitationIsGeneric}
The sequence $\seq{\lambda_i}{i<\tilde{\alpha}}$ of critical points of the imitation of $\P$ gives rise to an $M_{\tilde{\alpha}}$-generic filter on $\pi_{0,\tilde{\alpha}}(\P)$.
\end{thm}

The proof of this theorem makes use of a Mathias-like characterization of genericity of generalized \Prikry{} sequences, which says, again in our special case (but there is a general version) that a sequence $\seq{\gamma_i}{i<\alpha}$ with $\gamma_i\in[\sup_{j<i}\kappa_j,\kappa_i)$ for all $i<\alpha$ is generic over $\V$ if for every sequence $\seq{X_i}{i<\alpha}$ in $\V$ such that $X_i\in U_i$ for all $i$, $\gamma_i\in X_i$ for all but at most finitely many $i<\alpha$ (see \cite[Theorem 1]{Fuchs:COPS}). Knowing this characterization, it is straightforward to check that Theorem \ref{thm:CriticalSequenceOfImitationIsGeneric} holds, that is, that the critical sequence is generic over the limit model.

One can do a product analysis of generalized \Prikry{} forcing in the obvious way. For $i\le j\le\alpha$ and a condition $p=\kla{s,T}\in\P$, write $p\rest[i,j)$ for $\kla{s\rest[i,j),T\rest[i,j)}$, and let $\P\rest[i,j)=\{p\rest[i,j)\st p\in\P\}$, equipped with the obvious ordering. Then, $\P\isomorphic\P\rest[0,i)\times\P\rest[i,\alpha)$. Of course, $\P\rest[i,j)$ can itself be viewed as a generalized \Prikry{} forcing (by shifting the index set $[i,j)$ to $[0,j-i)$), and so, it makes sense to refer to the imitation iteration of $\P\rest[i,j)$.

The goal is to show that the embedding $\pi_{0,\tilde{\alpha}}:\V\To M_{\tilde{\alpha}}$ is the Boolean ultrapower embedding from $\V$ into its Boolean ultrapower $\check{\V}_U$, where $U$ is the pullback of the $\pi_{0,\tilde{\alpha}}(\P)$-generic filter generated over $M_{\tilde{\alpha}}$ by the critical sequence. This is less obvious in the present case than it was in the case of \Prikry{} forcing or Magidor forcing. We will approach the full result in a few smaller steps.

\begin{defn}
$\P$ is \emph{short} if $\alpha<\kappa_0$. It is \emph{medium} if $\kappa_0\le\alpha<\theta=\sup_{i<\alpha}\kappa_i$. It is \emph{long} if $\alpha=\theta$.
\end{defn}

Of course, $\alpha$ can not be greater than $\theta$.

In the following, let us fix $\P$ and its imitation iteration with models $\seq{M_i}{i\le\tilde{\alpha}}$ and embeddings $\seq{\pi_{i,j}}{i\le j\le\tilde{\alpha}}$. Let $\B$ be the Boolean algebra of $\P$. Let $G_{\vlambda}$ be the filter on $\pi_{0,\tilde{\alpha}}(\B)$ generated by $\vlambda$ over $M_{\tilde{\alpha}}$, and let $U=\pi_{0,\tilde{\alpha}}^{-1}``G_{\vlambda}$ be the pullback of $G_\vlambda$, an ultrafilter on $\B$. Let $j:\V\To_U\check{\V}_U$.

\begin{lem}
\label{lem:IfPisShortThenTheImitationIsABooleanUltrapower}
If $\P$ is short, then $j=\pi_{0,\tilde{\alpha}}$, $M_{\tilde{\alpha}}=\check{\V}_U$, and $M_{\tilde{\alpha}}[\vlambda]=\V^\B/U$.
\end{lem}

\begin{proof} This is the case in which the argument proceeds very much as in the case of \Prikry{} or Magidor forcing. By Theorem \ref{thm:HamkinsSeaboldCriterionForBeingBooleanUltrapower}, we have to show that every member $a$ of $M_{\tilde{\alpha}}$ can be written as $a=\pi_{0,\tilde{\alpha}}(f)(b_A)$, for some maximal antichain $A\sub\P$, where as before, $b_A$ is the unique element of the $M_{\tilde{\alpha}}$-generic filter that belongs to $\pi_{0,\tilde{\alpha}}(A)$. Clearly, there is a function $f$ such that $a=\pi_{0,\tilde{\alpha}}(f)(\lambda_{i_0},\ldots,\lambda_{i_{n-1}})$, for some $n<\omega$ and $i_0<i_1<\ldots<i_{n-1}<\tilde{\alpha}$. Note that since $\P$ is short, $\tilde{\alpha}=\alpha$, $W_i=\pi_{0,i}(U_i)$ and $\lambda_i=\pi_{0,i}(\kappa_i)$. So we can let $A$ be the maximal antichain in $\P$ consisting of all conditions $\kla{s,T}$ with $\dom(s)=\{i_0,\ldots,i_{n-1}\}$ and $T(i)=\kappa_i$ for all $i\in\alpha\ohne\{i_0,\ldots,i_{n-1}\}$. Define $g:A\To\V$ by setting $g(\kla{s,T})=f(s(i_0),\ldots,s(i_{n-1}))$. We claim that $a=\pi_{0,\tilde{\alpha}}(g)(b_A)$. To see this, let $b_A=\kla{s,T}$. Since $\kla{s,T}\in j(A)$, the domain of $s$ must be $\pi_{0,\tilde{\alpha}}(\{i_0,\ldots,i_{n-1}\})=\{i_0,\ldots,i_{n-1}\}$. And since $\kla{s,T}$ belongs to the filter generated by $\vlambda$, it follows that $s(i_j)=\lambda_{i_j}$, for $j<n$. By definition of $g$, then, $\pi_{0,\tilde{\alpha}}(g)(b_A)=\pi_{0,\tilde{\alpha}}(g)(\kla{s,T})=\pi_{0,\tilde{\alpha}}(f)(\lambda_{i_0},\ldots,\lambda_{i_{n-1}})=a$.
\end{proof}

Let us now work our way towards the case that $\P$ is medium. The following is a technical lemma that will be useful.

\begin{lem}
\label{lem:Crux}
Let $a\in M_{\tilde{\alpha}}$, and suppose there is an $i<\tilde{\alpha}$, an $A\in M_i$ which is a maximal antichain in $\pi_{0,i}(\P)\rest[i,\alpha_i)$ in $M_i$ and a function $f:A\To M_i$, $f\in M_i$, such that $a=\pi_{i,\tilde{\alpha}}(f)(b_A)$. Then there is a maximal antichain $A^*\sub\P$ and a function $f^*:A^*\To\V$ such that $a=\pi_{0,\tilde{\alpha}}(f^*)(b_{A^*})$.
\end{lem}

\noindent\emph{Note:} The coordinate $i$ can be read off of $A$, since $i$ is the minimum of the union of the domains of $s$ and $T$, for any $\kla{s,T}\in A$. The notation $b_A$ then makes sense: it is the unique member of $\pi_{i,\tilde{\alpha}}(A)$ that belongs to the filter on $\pi_{i,\tilde{\alpha}}(\P)\rest[i,\tilde{\alpha})$ generated by $\seq{\lambda_j}{i\le j<\tilde{\alpha}}$, which is generic for that partial order over $M_{\tilde{\alpha}}$. Note that the critical point of $\pi_{i,\tilde{\alpha}}$ is $\lambda_i=\pi_{0,i}(\vkappa)_i>\sup_{j<i}\pi_{0,i}(\vkappa)_j\ge i$, so $\pi_{i,\tilde{\alpha}}(i)=i$.

\begin{proof}
Let $j$ be least such that there is an $\tA\in M_j$, such that $\tA$ is a maximal antichain in $\pi_{0,j}(\P)\rest[j,\alpha_j)$, and an $\tf:\tA\To M_j$, $\tf\in M_j$, such that $a=\pi_{j,\tilde{\alpha}}(\tf)(b_{\tA})$. Fix such $\tA$ and $\tf$. We claim that $j=0$. Clearly, there is such a $j$, and $j\le i$, as witnessed by $A$ and $f$. If $j$ is not $0$, then $j$ is either a successor or a limit ordinal.

\noindent\emph{Case 1:} $j$ is a successor ordinal, say $j=\gamma+1$.

Let $\vA=\seq{A_\xi}{\xi<\lambda_\gamma}$, $\vec{f}=\seq{f_\xi}{\xi<\lambda_\gamma}\in M_\gamma$ so that $[\vA]_{W_\gamma}=\tilde{A}$ and $[\vf]_{W_\gamma}=\tilde{f}$. Recall that $\crit(\pi_{\gamma,\gamma+1})=\lambda_\gamma$ is greater than $\gamma$, hence it is greater than $\gamma+1$ as well, so we may assume that for all $\xi<\lambda_\gamma$, $A_\xi$ is a maximal antichain in $\pi_{0,\gamma}(\P)\rest[\gamma+1,\alpha_\gamma)$ and $f_\xi:A_\xi\To M_\gamma$. Set
\[A^*=\{\kla{s\cup\{\kla{\gamma,\xi}\},T}\st\xi<\lambda_\gamma\ \text{and}\ \kla{s,T}\in A_\xi\}\]
and define $f^*:A^*\To M_{\gamma}$ by setting
\[f^*(\kla{s,T})=f_{s(\gamma)}(\kla{s,T}\rest[\gamma+1,\alpha_\gamma))\]
Then $A^*$ is a maximal antichain in $\pi_{0,\gamma}(\P)\rest[\gamma,\alpha_\gamma)$: it is obvious that it is an antichain. To see that it is maximal, suppose $\kla{s,T}\in\pi_{0,\gamma}(\P)\rest[\gamma,\alpha_\gamma)$ were incompatible with all members of $A^*$. W.l.o.g., we may assume that $\gamma\in\dom(s)$. Let $s(\gamma)=\xi$. Then $\kla{s,T}\rest[\gamma+1,\alpha_\gamma)$ would be incompatible (with respect to $\pi_{0,\gamma}(\P)\rest[\gamma+1,\alpha_\gamma)$) with all members of $A_\xi$, which contradicts the maximality of $A_\xi$.

We claim that $\pi_{\gamma,\tilde{\alpha}}(f^*)(b_{A^*})=a$. To see this, let $b_{A^*}=\kla{s,T}$. Since every member $\kla{t,H}\in A^*$ has $\gamma$ in its domain, the same is true for $\pi_{\gamma,\tilde{\alpha}}(A^*)$. Since $b_{A^*}$ belongs to the filter generated by $\seq{\lambda_\zeta}{\gamma\le\zeta<\tilde{\alpha}}$, it must be the case that $s(\gamma)=\lambda_\gamma$. But this means that $\kla{s,T}\rest[\gamma+1,\pi_{0,\tilde{\alpha}}(\alpha))\in\pi_{\gamma,\tilde{\alpha}}(\vA)_{\lambda_\gamma}$, because of the way $A^*$ was constructed from the $A_\xi$ ($\xi<\lambda_\gamma$). But $\pi_{\gamma,\tilde{\alpha}}(\vA)_{\lambda_\gamma}=\pi_{\gamma+1,\tilde{\alpha}}(\pi_{\gamma,\gamma+1}(\vA)(\lambda_\gamma))=\pi_{\gamma+1,\tilde{\alpha}}([\vA]_{W_\gamma})=\pi_{\gamma+1,\tilde{\alpha}}(\tilde{A})$.
This means that $b_{A^*}\rest[\gamma+1,\tilde{\alpha})=b_{\tA}$. By the way $f^*$ was defined, we get:
\begin{ea*}
\pi_{\gamma,\tilde{\alpha}}(f^*)(b_{A^*})&=&(\pi_{\gamma,\talpha}(\vf)_{\lambda_\gamma})(b_{A^*}\rest[\gamma+1,\talpha))\\
&=&\pi_{\gamma+1,\talpha}([\vf]_{W_\gamma})(b_{\tA})\\
&=&\pi_{\gamma+1,\talpha}(\tf)(b_{\tA})\\
&=&a.
\end{ea*}
This contradicts the minimality of $j=\gamma+1$.

\noindent\emph{Case 2:} $j$ is a limit ordinal.

Since $M_j$ is the direct limit of the earlier structures, there is a $\gamma<j$ such that $\tf$ and $\tA$ are in the range of $\pi_{\gamma,j}$. Let $\barf,\bA$ be the preimages. Since $j$ is definable from $\tA$, it follows that $j$ is in the range of $\pi_{\gamma,j}$ as well. Let $\bj=\pi_{\gamma,j}^{-1}(j)$. It follows that $\bj\ge\gamma$, because if $\bj<\gamma$, then since $\gamma<\lambda_\gamma=\crit(\pi_{\gamma,j})$, it would follow that $j=\pi_{\gamma,j}(\bj)=\bj<\gamma<j$. So we have that $\bA\sub\pi_{0,\gamma}(\P)\rest[\bj,\alpha_\gamma)$ is a maximal antichain, and $\barf:\bA\To M_\gamma$. We can now easily expand $\bA$ to a maximal antichain $A^*$ in $\pi_{0,\gamma}(\P)\rest[\gamma,\alpha_\gamma)$ by setting
\begin{ea*}
A^*=\{\kla{s,T}\in\pi_{0,\gamma}(\P)\rest[\gamma,\alpha_\gamma)&|&\kla{s,T}\rest[\bj,\alpha_\gamma)\in\bA\ \text{and for all}\ \xi\in[\gamma,\bj),\\ &&T(\xi)=\pi_{0,\gamma}(\vkappa)_\xi\}
\end{ea*}
Similarly, $\barf$ can be expanded to $f^*$, so as to act on $A^*$, so that $f^*:A^*\To M_\gamma$, by setting
\[f^*(p)=\barf(p\rest[\bj,\alpha_\gamma))\]
It is straightforward to check that $A^*$ is a maximal antichain in $\pi_{0,\gamma}(\P)\rest[\gamma,\alpha_\gamma)$: it is clearly an antichain, and it is maximal because if $\kla{s,T}$ were incompatible with every member of $A^*$, then since the conditions in $A^*$ are trivial below $\bj$, it would follow that $\kla{s,T}\rest[\bj,\alpha_\gamma)$ is incompatible with every member of $\bA$, which would contradict the maximality of $\bA$.

Of course, our next claim is that $\pi_{\gamma,\talpha}(f^*)(b_{A^*})=a$. To see this, let $\kla{s,T}=b_{\tA}$. Let $\kla{s,T'}\in\pi_{0,\talpha}(\P)\rest[\gamma,\talpha)$ be defined by letting $T'\rest[j,\talpha)=T$, and for $\xi\in[\gamma,j)$, $T'(\xi)=\pi_{0,j}(\vkappa)_\xi$. Then $\kla{s,T'}\in\pi_{\gamma,\talpha}(\bA)$, and $\kla{s,T'}$ belongs to the generic filter generated by $\seq{\lambda_\xi}{\gamma\le\xi<\talpha}$. So $\kla{s,T'}=b_{A^*}$. Hence, $b_{A^*}\rest[j,\talpha)=b_{\tA}$, and so,
\[\pi_{\gamma,\talpha}(f^*)(b_{A^*})=\pi_{\gamma,\talpha}(\barf)(b_{\tA})=\pi_{j,\talpha}(\tf)(b_{\tA})=a\]
as desired.

Again, we have reached a contradiction to the minimality of $j$. So the only possibility is that $j=0$, and this proves the lemma. \end{proof}

\begin{lem}
\label{lem:IfPisMediumThenTheImitationIsABooleanUltrapower}
If $\P$ is medium, then $j=\pi_{0,\tilde{\alpha}}$, $M_{\tilde{\alpha}}=\check{\V}_U$, and $M_{\tilde{\alpha}}[\vlambda]=\V^\B/U$.
\end{lem}

\begin{proof} As in the proof of Lemma \ref{lem:IfPisShortThenTheImitationIsABooleanUltrapower}, we will use Theorem \ref{thm:HamkinsSeaboldCriterionForBeingBooleanUltrapower}. So we have to show that every member $a$ of $M_{\tilde{\alpha}}$ can be written as $a=\pi_{0,\tilde{\alpha}}(f)(b_A)$, for some maximal antichain $A\sub\P$. So let $a\in M_\talpha$ be given.

Since $\P$ is medium, there is a $\gamma<\alpha$ such that $\alpha<\kappa_\gamma$. Note that $\gamma>0$, as $\P$ is medium, not short. So by elementarity, there is a $\gamma<\talpha=\pi_{0,\talpha}(\alpha)$ such that $\talpha<\pi_{0,\talpha}(\vkappa)_\gamma$. Let $i$ be the least such $\gamma$. Then $\pi_{0,i}(\alpha)<\pi_{0,i}(\vkappa)_i$. This is because if $\pi_{0,i}(\alpha)\ge\pi_{0,i}(\vkappa)_i$, then, applying $\pi_{i,\talpha}$, we would get
\[\talpha=\pi_{0,\talpha}(\alpha)\ge\pi_{i,\talpha}(\pi_{0,i}(\vkappa)_i)=\pi_{0,\talpha}(\vkappa)_{\pi_{i,\talpha}(i)}=\pi_{0,\talpha}(\vkappa)_i\]
because $i<\crit(\pi_{i,\talpha})$. So $\talpha\ge\pi_{0,\talpha}(\vkappa)_i$, contradicting the choice of $i$.

This means that the forcing $\pi_{0,i}(\P)[i,\alpha_i)$ is short (slightly abusing notation), where we again write $\alpha_i$ for $\pi_{0,i}(\alpha)$. But notice that the tail of the iteration after $i$, $\seq{M_\gamma}{i\le\gamma\le\tilde{\alpha}}$ is the imitation of $\pi_{0,i}(\P)\rest[i,\alpha_i)$. So we can apply Lemma \ref{lem:IfPisShortThenTheImitationIsABooleanUltrapower}, and we get that $a=\pi_{i,\talpha}(f)(b_A)$, for a maximal antichain $A$ in $\pi_{0,i}(\P)\rest[i,\alpha_i)$ and a function $f:A\To M_i$, where $f,A\in M_i$. But then Lemma \ref{lem:Crux} kicks in, telling us that there is a maximal antichain $A^*$ in $\P$ and a function $f^*:A^*\To\V$ such that $a=\pi_{0,\talpha}(f^*)(b_{A^*})$. Since $a$ was an arbitrary member of $M_{\talpha}$, this proves the lemma. \end{proof}

\begin{lem}
\label{lem:IfPisLongThenTheImitationIsABooleanUltrapower}
If $\P$ is long, then $j=\pi_{0,\tilde{\alpha}}$, $M_{\tilde{\alpha}}=\check{\V}_U$, and $M_{\tilde{\alpha}}[\vlambda]=\V^\B/U$.
\end{lem}

\begin{proof}
Let $a\in M_{\talpha}$. According to Theorem \ref{thm:HamkinsSeaboldCriterionForBeingBooleanUltrapower} again, we have to find $f^*,A^*$ such that $A^*$ is a maximal antichain in $\P$, $f^*:A^*\To\V$, and $a=\pi_{0,\talpha}(f^*)(b_{A^*})$. Since $\P$ is long, $\talpha$ is a limit ordinal, and so, there is an $i<\talpha$ and an $\ba\in M_i$ such that $\pi_{i,\talpha}(\ba)=a$. Let $\P_i=\pi_{0,i}(\P)\rest[i,\alpha_i)$, let $A=\{\eins_{\P_i}\}$, and let $f:A\To M_i$ be defined by $f(\eins_{\P_i})=\ba$. Clearly then, $\pi_{i,\talpha}(f)(b_A)=\pi_{i,\talpha}(\ba)=a$. But then, Lemma \ref{lem:Crux} implies that there are $f^*$ and $A^*$ with the desired properties.
\end{proof}

Note that the proof of the previous lemma only made use of the fact that $\alpha$ is a limit ordinal under its assumption.

Now, Lemmas \ref{lem:IfPisShortThenTheImitationIsABooleanUltrapower}, \ref{lem:IfPisMediumThenTheImitationIsABooleanUltrapower} and \ref{lem:IfPisLongThenTheImitationIsABooleanUltrapower} yield:

\begin{thm}
If $\P=\P_{\vU,\veta}$ is a generalized \Prikry{} forcing of length $\alpha$, with $\eta_i=1$ for all $i<\alpha$, $\B$ is the Boolean algebra of $\P$ and if $\vM$ is the imitation iteration of $\P$ with embeddings $\vpi$, and if $U$ is the pullback of the filter on $\pi_{0,\talpha}(\B)$ generated by the critical sequence, and $j:\V\To\check{V}_U$ is the Boolean ultrapower, then
\begin{itemize}
\item $j=\pi_{0,\tilde{\alpha}}$,
\item $M_{\tilde{\alpha}}=\check{\V}_U$, and
\item $M_{\tilde{\alpha}}[\vlambda]=\V^\B/U$.
\end{itemize}
\end{thm}

The result should hold without the restriction to the case where $\eta_i=1$ for all $i<\alpha$. We made this assumption merely to keep the notation simple.

\section{The \Bukovsky-Dehornoy Phenomenon, simple antichains and skeletons}
\label{section:Bukovsky-Dehornoy}

It is by now a well-known fact that when one iterates a normal measure $\omega$ many times, and then intersects the finite iterates, the resulting model, let's call it the \emph{intersection model}, is the generic extension of the direct limit by the \Prikry{}-generic filter generated by the critical sequence. The history of this fact is as follows: \Bukovsky{} showed that the intersection model is a generic extension of the direct limit model in \cite{Bukovsky1973:ChangingCofinalityAlternativeProof}, but it was left open what forcing led from the direct limit model to the intersection model. Dehornoy then showed in \cite{Dehornoy:IteratedUltrapowersAndPrikryForcing} that the forcing is \Prikry{} forcing (and he showed much more general results in that paper). In \cite{Bukovsky1977:IteratedUltrapowersAndPrikryForcing}, \Bukovsky{} showed that the intersection model is the same as the Boolean model $\V^\B/U$ and that the direct limit model is the same as $\check{\V}_U$.

In this section, we investigate whether this phenomenon can be extended to the Boolean ultrapower setting, or whether Boolean ultrapowers can be used to explain it. The attractive scenario is that the models $M_A$ correspond to the finite iterates in the case of \Prikry{} forcing, the direct limit model corresponds to $\check{\V}_U$,  the intersection model corresponds to the intersection of the models $M_A$, and the generic extension of the direct limit model corresponds to $\V^\B/U$. In short, the phenomenon we are interested in is
\[\bigcap_A M_A=\V^\B/U\]
Since in general, $\V^\B/U=\check{V}_U[G]$, this equation says it all, and if it holds, then we say that the \Bukovsky-Dehornoy phenomenon arises with $U$.

Recall that if $A\sub\B$ is a maximal antichain and we form $M_A$, the ultrapower of $\V$ by the ultrafilter $U_A$ on $A$, then we denote the canonical embedding from $\V$ into $M_A$ by $\pi_{0,A}$, and the direct limit embedding is denoted $\pi_{A,\infty}:M_A\To\check{\V}_U$.

\subsection{A condition implying the \Bukovsky-Dehornoy phenomenon}

Before continuing our analysis of situations when the \Bukovsky-Dehornoy phenomenon arises, let us note that it is not universal. The following was observed jointly by the first author and Joseph van Name.

\begin{obs}
If $\kappa$ is a measurable cardinal such that $2^\kappa=\kappa^+$, then there is an ultrafilter $U$ on the complete Boolean algebra $\B$ of $\Add(\kappa)$ such that the \Bukovsky-Dehornoy phenomenon fails for $U$.
\end{obs}

\begin{proof}
This is a simple consequence of Corollary 81 in \cite{HamkinsSeabold:BULC}. Letting $\mu$ be a normal ultrafilter on $\kappa$, that corollary says that under the assumptions of the observation, there is an ultrafilter $U$ on $\B$ such that the Boolean ultrapower by $U$ is equal to the ultrapower of $\V$ by $\mu$, and the embeddings $j_U$ and $j_\mu$ are equal. Now, if $A$ is a maximal antichain in $\B$, the embedding $j_U=j_\mu$ factors as $j_\mu=\pi_{A,\infty}\circ\pi_{0,A}$. Since $\mu$ is normal, and hence minimal in the Rudin-Keisler order on $\kappa$-complete ultrafilters, it follows that this factoring is trivial, i.e., one factor is the identity. Now it cannot be that for every maximal antichain, $\pi_{0,A}$ is the identity, because then, it would follow that whenever $B$ is a maximal antichain refining $A$, $\pi_{A,B}$ is the identity, because $\pi_{0,B}=\pi_{A,B}\pi_{0,A}$. But then $\pi_{A,\infty}$ would have to be the identity as well, for every $A$, being the direct limit of this system. So, let $A$ be a maximal antichain such that $\pi_{0,A}$ is not the identity. Then $\pi_{A,\infty}$ is the identity, and so, $M_A=\check{\V}_U$. It follows easily that the intersection model is $\bigcap_A M_A=\check{\V}_U$, but of course, $[\dot{G}]_U\in(\V^\B/U)\ohne\check{\V}_U$, since $\B$ is non-atomic.
\end{proof}

We will explore failures of the \Bukovsky-Dehornoy phenomenon further in a future project. For now, the goal is to find criteria that ensure that it holds. It will be crucial to analyze when the embedding $\pi_{A,\infty}:M_A\To\check{\V}_U$ is itself a Boolean ultrapower embedding.

\begin{defn}
\label{defn:G_A}
If $A$ is a maximal antichain in $\B$, then let $G_A=\pi_{A,\infty}^{-1}``G$.
\end{defn}

\begin{thm}
\label{thm:ImportanceOfG_A}
Let $A\sub\B$ be a maximal antichain. Then
\begin{enumerate}
\item
\label{item:BooleanModelsAreTheSame}
$\V^\B/U\isomorphic M_A^{\pi_{0,A}(\B)}/G_A$.
\item
\label{item:BooleanUltrapowersAreTheSame}
$\pi_{A,\infty}:M_A\To_{G_A}\check{V}_U$,

which means that the direct limit embedding $\pi_{A,\infty}:M_A\To\check{V}_U$ is the same as the Boolean ultrapower embedding $j^{(M_A,\pi_{0,A}(\B),G_A)}:M_A\To\Ult(M_A,\pi_{0,A}(\B),G_A)$.
\item
\label{item:GenericsAreTheSame}
$\chi\rest G:G\isomorphism\tilde{G}$,

where $\tilde{G}=[\pi_A(\dot{G})]_{G_A}$ and $\chi:\V^\B/U\To M_A^{\pi_A(\B)}/G_A$ is the isomorphism from \ref{item:BooleanModelsAreTheSame}.
\end{enumerate}
\end{thm}

\begin{proof} Let us first show \ref{item:BooleanModelsAreTheSame}.

We define an embedding $\chi:\V^\B/U\To M_A^{\pi_A(\B)}/G_A$, and show that it is onto. Define $\chi$ by setting
\[\chi([\tau]_U)=[\pi_A(\tau)]_{G_A}\]
First, let us show that this is a correct definition. So suppose $x=[\tau]_U=[\tau']_U$. This means that $\BV{\tau=\tau'}^\B\in U$. Since $j``U\sub G$, this implies that $\BV{j(\tau)=j(\tau')}^{\check{V}_U,j(\B)}\in G$. Since $j=\pi_{A,\infty}\circ\pi_A$, we can pull back via $\pi_{A,\infty}^{-1}$, giving $\BV{\pi_A(\tau)=\pi_A(\tau')}^{M_A,\pi_A(\B)}\in G_A$, as $G_A=\pi_{A,\infty}^{-1}``G$. This means that $[\pi_A(\tau)]_{G_A}=[\pi_A(\tau')]_{G_A}$, and thus shows that $\chi$ is well-defined. Replacing ``$\tau=\tau'$'' with any desired formula $\varphi(\vec{\tau})$ in this argument shows now that $\chi$ is an elementary embedding.\\

\noindent\emph{Note:} The above proof shows that whenever $U'$ is an ultrafilter on $\pi_A(\B)$ such that $\pi_A``U\sub U'$, the stipulation $\chi([\tau]_U)=[\pi_A(\tau)]_{U'}$ is a correct definition of an elementary embedding from $\V^\B/U$ to $M_A^{\pi_A(\B)}/U'$. Note also that it is always the case that $j^{-1}``G=U$. Applying this to $M_A$, $\pi_A(\B)$, this means that if $\pi_{A,\infty}:M_A\To_{U'}\check{V}_U$ and $G=[\dot{G}]_{U'}$, then it has to be the case that $U'=\pi_{A,\infty}^{-1}``G$. \\

Let us now show that $\chi$ is onto. Let $y=[\nu]_{G_A}$ be an arbitrary element of $M_A^{\pi_A(\B)}/G_A$, for some $\nu\in M_A^{\pi_A(\B)}$. Let $\nu=[f]_{U_A}$, where $f:A\To\V^\B$, $f\in\V$. Let $\tau=\mix(f)$ be the canonical name resulting from mixing the names $f(a)$, for $a\in A$. So
\[\text{for every}\ a\in A,\quad  a\le\BV{\tau=f(a)}\]
Let $x=[\tau]_U$. We claim that $\chi(x)=y$.

By definition, $\chi(x)=\chi([\tau]_U)=[\pi_A(\tau)]_{G_A}$. And $y=\bigl[[f]_{U_A}\bigr]_{G_A}$. So in order to show that $\chi(x)=y$, we have to show that
\[[\pi_A(\tau)]_{G_A}=\bigl[[f]_{U_A}\bigr]_{G_A}\]
This is equivalent to
\[\BV{\pi_A(\tau)=[f]_{U_A}}^{M_A,\pi_A(\B)}\in G_A\]
Let $h:A\To\B$ be defined in $\V$ by
\[h(a)=\BV{\tau=f(a)}\]
So
\[[h]_{U_A}=\BV{\pi_A(\tau)=[f]_{U_A}}^{M_A,\pi_A(\B)}\]
and we have to show that $[h]_{U_A}\in G_A$. By definition of $G_A$, this is equivalent to saying that
\[\pi_{A,\infty}([h]_{U_A})\in G\]
By definition, $\pi_{A,\infty}([h]_{U_A})=[h]_U=j(h)(b_A)=\BV{j(\tau)=j(f)(b_A)}^{\check{\V}_U,j(\B)}$. But by definition of $\tau$ as mixing $f$, in $\V$, for every $a\in A$, $a\le\BV{\tau=f(a)}$, so by elementarity of $j$,
\[\text{for every}\ a\in j(A), \quad a\le\BV{j(\tau)=j(f)(a)}^{\check{\V}_U,j(\B)}\]
Since $b_A\in j(A)$, it follows that $b_A\le\BV{j(\tau)=j(f)(b_A)}^{\check{\V}_U,j(\B)}$. So, since also, $b_A\in G$, it follows that $\BV{j(\tau)=j(f)(b_A)}^{\check{\V}_U,j(\B)}\in G$. This shows that $\chi$ is onto.

Let us now show part \ref{item:BooleanUltrapowersAreTheSame}. Writing
\[j':(M_A,\pi_A(\B),G_A)\To_{G_A}\Ult(M_A,\pi_A(\B),G_A)\]
we want to show that $\chi\circ\pi_{A,\infty}=j'$. Since $\chi$ is the identity, this implies that $j'=\pi_{A,\infty}$.

Let $x=[f]_{U_A}\in M_A$, $f:A\To\V$, $f\in\V$. Then
\begin{ea*}
  \chi(\pi_{A,\infty}(x))&=&\chi(\pi_{A,\infty}([f]_{U_A})\\
  &=&\chi([f]_U)\\
  &=&\chi(j(f)(b_A))\\
  &=&\chi((j(\tilde{f})(b_A))^G),\qquad\text{where $\tilde{f}(a)=(f(a))\check{}$ for all $a\in A$}\\
  &=&\chi([\mix(\tilde{f})]_U)\\
  &=&[\pi_A(\mix(\tilde{f}))]_{G_A}
\end{ea*}%
by definition of $\chi$. We want to show that this is the same as $j'(x)$, which, in turn, is $[\check{x}]_{G_A}$. So we have to show:
\[[\pi_A(\mix(\tilde{f}))]_{G_A}=\bigl[([f]_{U_A})\check{}\bigr]_{G_A}\]
To see this, first, observe that it is equivalent to
\[\BV{\pi_A(\mix(\tilde{f}))=([f]_{U_A})\check{}}^{M_A,\pi_A(\B)}\in G_A\]
As in the previous proof, define $h:A\To\B$ by
\[h(a)=\BV{\mix(\tilde{f})=(f(a))\check{}}^\B\]
So we have to show that $\pi_{A,\infty}([h]_{U_A})\in G$. But \[\pi_{A,\infty}([h]_{U_A})=[h]_U=j(h)(b_A)=\BV{j(\mix(\tilde{f}))=(j(f)(b_A))\check{}}^{\check{\V}_U,j(\B)}\] Since $b_A\in j(A)$ and $j(\mix(\tilde{f}))$ mixes the $(j(f)(a))\check{}$ names, it follows that
\[b_A\le\BV{j(\mix(\tilde{f}))=(j(f)(b_A))\check{}}^{\check{\V}_U,j(\B)}\]
and since $b_A\in G$, it follows that \[\BV{j(\mix(\tilde{f}))=(j(f)(b_A))\check{}}^{\check{\V}_U,j(\B)}\in G\]
This completes the proof.

To see \ref{item:GenericsAreTheSame}, we can use the fact that $\chi$ is onto, and argue:
\begin{ea*}
[\sigma]_U\in G&\iff\BV{\sigma\in\dot{G}}\in U\\
&\implies&(\check{\V}_U,j(\B))\models\BV{j(\sigma)\in j(\dot{G})}\in G\\
&\implies&(M_A,G_A)\models\BV{\pi_A(\sigma)\in \pi_A(\dot{G})}\in G_A\\
&\iff&(M_A,G_A)\models[\pi_A(\sigma)]_{G_A}\in[\pi_A(\dot{G})]_{G_A}\\
&\iff&\chi([\sigma]_U)\in\tilde{G}.
\end{ea*}

For the converse, we can argue similarly:
\begin{ea*}
[\sigma]_U\notin G&\iff\BV{\sigma\notin\dot{G}}\in U\\
&\implies&(\check{\V}_U,j(\B))\models\BV{j(\sigma)\notin j(\dot{G})}\in G\\
&\implies&(M_A,G_A)\models\BV{\pi_A(\sigma)\notin \pi_A(\dot{G})}\in G_A\\
&\iff&(M_A,G_A)\models[\pi_A(\sigma)]_{G_A}\notin[\pi_A(\dot{G})]_{G_A}\\
&\iff&\chi([\sigma]_U)\notin\tilde{G}.
\end{ea*}
\end{proof}

For the remainder of this section, except for the subsection in the end, we will assume that $\V^\B/U$ is well-founded, and hence, we can take all the models $M_A$, $\V^\B/U$ and $\check{\V}_U$ to be transitive. Without this assumption, it does not make sense to consider the intersection $\bigcap_A M_A$, for example, and so, taken literally, the \Bukovsky-Dehornoy phenomenon does not make sense. We will explore a version of the phenomenon for ill-founded ultrapowers in the subsection at the end of the present section.

Note that the embedding $j^{M_A,\pi_{0,A}(\B),G_A}$ (which is the same as the embedding $\pi_{A,\infty}$) is internal to $\calM_A$ if $G_A\in M_A$. This makes the assumption $G_A\in M_A$ a natural one.

\begin{defn}
An antichain $A\sub\B$ is \emph{simple} (with respect to $U$) if $G_A\in M_A$ (when the ultrapowers are formed with respect to $U$).
\end{defn}

We will often suppress the reference to the ultrafilter $U$ used for the formation of the (Boolean) ultrapowers, but of course, the meaning of $G_A$ and $M_A$ depends on $U$.

\begin{question}
If $B\le^*A$ are maximal antichains in $\B$ and $B$ is simple, then does it follow that $A$ is simple?
\end{question}

This is clearly the case if $G_A\in\ran(\pi_{A,B})$, and in that case, $G_A=\pi_{A,B}^{-1}(G_B)\in M_A$. This is because
\begin{ea*}
x\in\pi_{A,B}^{-1}(G_B)&\iff&\pi_{A,B}(x)\in G_B\\
&\iff&\pi_{B,\infty}(\pi_{A,B}(x))\in G\\
&\iff&\pi_{A,\infty}(x)\in G\\
&\iff&x\in G_A.
\end{ea*}%
If $A$ is simple, then not only is the embedding $\pi_{A,\infty}=j^{(M_A,\pi_{0,A}(\B),G_A)}$ internal to $\calM_A$, but the entire construction of $\V^\B/U=M_A^{\pi_{0,A}(\B)}/G_A$ is internal to $\calM_A$. This gives the following corollary.

\begin{cor}
\label{cor:EasyDirectionForSimpleAntichainsLocal}
If $A$ is simple, then $\V^\B/U\sub M_A$.
\end{cor}

What we are aiming to show is that
\[\bigcap_A M_A=\V^\B/U\]
or to determine under which conditions this is true. The previous corollary may be used for the direction from right to left, giving

\begin{lem}
\label{lem:EasyDirectionForSimpleAntichainsGlobal}
If every maximal antichain $A\sub\B$ is simple, then
\[\V^\B/U\sub\bigcap_A M_A.\]
More generally,
\[\V^\B/U\sub\bigcap_{A\ \text{is simple}}M_A.\]
\end{lem}

Towards formulating a sufficient criterion for the other inclusion, we need the following definition.

\begin{defn}
\label{def:Skeleton}
A \emph{skeleton} for ($\B ,U$) is a set $\mathfrak{A}$ of maximal antichains in $\B$ such that
\begin{enumerate}
\item $\mathfrak{A}$ is directed under refinement,
\item $j``\mathfrak{A}\in\V^\B/U$,
\item $\check{V}_U=\{j(f)(b_A)\st A\in\mathfrak{A}\ \text{and}\ f:A\To\V\}$.
\end{enumerate}
A skeleton $\mathfrak{A}$ is \emph{simple} if every $A\in\mathfrak{A}$ is simple.
\end{defn}

Note that if $\delta=|\mathfrak{A}|$, then $j``\mathfrak{A}\in\V^\B/U$ iff $j``\delta\in\V^\B/U$, and that is the case iff ${}^\delta(\V^\B/U)\sub\V^\B/U$ (for this last equivalence, see \cite[Theorem 28]{HamkinsSeabold:BULC}). In particular, if $\delta\le\crit(j)$, then $j``\mathfrak{A}\in\V^\B/U$.

We now show that the converse of Lemma \ref{lem:EasyDirectionForSimpleAntichainsGlobal} holds, assuming the existence of a simple skeleton.

\begin{thm}
\label{thm:SimpleSkeletonsMakeItWork}
If there is a simple skeleton $\mathfrak{A}$ for $\B$, $U$, then the \Bukovsky-Dehornoy phenomenon holds for the Boolean ultrapower of $\V$ by U. That is,
\[\V^\B/U=\bigcap_{A\in\mathfrak{A}}M_A\]
\end{thm}

\begin{proof} The direction from left to right follows from Corollary \ref{cor:EasyDirectionForSimpleAntichainsLocal}, since every $A$ in $\mathfrak{A}$ is simple.

For the converse, assume that the inclusion $\bigcap_{A\in\mathfrak{A}}M_A\sub\V^\B/U$ fails. Let $\tilde{a}\in\bigcap_{A\in\mathfrak{A}}M_A\ohne\V^\B/U$ be $\in$-minimal. It follows that $\tilde{a}\sub\V^\B/U$. Let $\gamma$ be the rank of $\tilde{a}$, and in $\V^\B/U$, let $f$ be a bijection between $\V_\gamma$ and its cardinality. Let $a=f``\tilde{a}$. We will show that $a\in\V^\B/U$, which implies that $\tilde{a}\in\V^\B/U$ as well, a contradiction.\footnote{The usual argument for reducing the claim to sets of ordinals does not apply, because while it is true that if $\tilde{a}\in M_A$, it can be coded by a set of ordinals in $M_A$, the resulting sets of ordinals in other models $M_B$ may be different.}

To show that $a\in\V^\B/U$, we are going to prove that $a$ is definable from parameters in $\V^\B/U$.

First, fix $A\in\mathfrak{A}$. Since $G_A\in M_A$, the function $\pi_{A,\infty}$ is definable in $M_A$, and so, the set $a_A:=\pi_{A,\infty}^{-1}``a\in M_A$. This means that there is a function $f_A:A\To\V$ such that $a_A=[f_A]_{U_A}$.

Now, let us consider the function
$A\mapsto f_A$ and denote it by $\vf$, so $f_A=\vf(A)=(\vf)_A$.

In $\V^\B/U$, we can define the map
$B\mapsto\tilde{b}_B,$
where $B$ is a maximal antichain in $j(\B)$ with $B\in\check{\V}_U$, and $\tilde{b}_B$ is the unique condition in $B\cap G$. For this definition, $G$ is needed as a parameter, so it can be carried out in $\V^\B/U$, but not in $\check{\V}_U$.

In the proof of the following claim, we will make use of the assumption that $j``\mathfrak{A}\in\V^\B/U$.

\claim{$(*)$}{$\alpha\in a$ iff there is an $A\in j``\mathfrak{A}$ such that for all $B\in j``\mathfrak{A}$ refining $A$,
\[\alpha\in j(\vf)_B(\tilde{b}_B).\]}

\prooff{$(*)$}
For the direction from left to right, assume that $\alpha\in a$. Pick $\bar{A}\in\mathfrak{A}$ so that $\alpha\in\ran(\pi_{\bar{A},\infty})$. Let $A=j(\bA)$. Then $A$ witnesses the right hand side of $(*)$, because if $B$ refines $A$ and is in $j``\mathfrak{A}$, then, letting $\bar{B}=j^{-1}(B)$, $\bar{B}$ refines $\bar{A}$, and hence, $\alpha\in\ran(\pi_{\bar{B},\infty})$. It follows that $\balpha=\pi_{\bar{B},\infty}^{-1}(\alpha)\in a_{\bar{B}}=[f_{\bar{B}}]_{U_{\bar{B}}}$. Applying $\pi_{\bar{B},\infty}$ to this fact, we get that $\alpha\in\pi_{\bar{B},\infty}([f_{\bar{B}}]_{U_{\bar{B}}})=j(f_{\bar{B}})(b_{\bar{B}})=j(\vf)_B(\tilde{b}_B)$.

For the direction from right to left, let $A$ be as in the claim. Let $\bar{A}=j^{-1}(A)$. Pick a refinement $\bar{B}$ of $\bar{A}$ in $\mathfrak{A}$ such that $\alpha\in\ran(\pi_{\bar{B},\infty})$, which is possible, since $\mathfrak{A}$ is a skeleton. By elementarity, $B:=j(\bar{B})$ refines $A$. So by our assumption, we have
\[\alpha\in j(\vf)_B(\tilde{b}_B) \]
But $j(\vf)_B(\tilde{b}_B)=j(f_{\bar{B}})(b_{\bar{B}})=\pi_{\bar{B},\infty}([f_{\bar{B}}]_{U_{\bar{B}}})=\pi_{\bar{B},\infty}(a_{\bar{B}})$.
So $\alpha\in\pi_{\bar{B},\infty}(a_{\bar{B}})$. This of course implies that $\alpha\in a$, because, letting $\pi_{\bar{B},\infty}(\balpha)=\alpha$, we have that $\pi_{\bar{B},\infty}(\balpha)\in\pi_{\bar{B},\infty}(a_{\bar{B}})$, so $\balpha\in a_{\bar{B}}=\pi_{\bar{B},\infty}^{-1}``a$, so $\alpha\in a$.
\qedd{(*)}

The point is now that we can define $a$ inside $\V^\B/U$: it is
\[\{\alpha\in\On\st\exists A\in j``\mathfrak{A}\forall B\in j``\mathfrak{A} (B\le^*A\implies
\alpha\in j(\vf)_B(\tilde{b}_B))\}\]
where $B\le^*A$ means that $B$ is a refinement of $A$.
\end{proof}

Hence, we get:

\begin{thm}
\label{thm:Bukovsky-DehornoyFromSimpleSkeleton}
Suppose $U$ is an ultrafilter on $\B$ such that the Boolean ultrapower by $U$ is well-founded. Suppose further that there is a simple skeleton for $(\B,U)$. Then
\[\V^\B/U=\bigcap_{A\ \text{simple}}M_A\]
If further, every maximal antichain in $\B$ is simple, then
\[\V^\B/U=\bigcap_{A}M_A\]
\end{thm}

\begin{proof} In both cases, the direction from left to right follows because if $A$ is simple, then $\V^\B/U$ is contained in $M_A$. For the direction from right to left, let $\mathfrak{A}$ be a simple skeleton. By Theorem \ref{thm:SimpleSkeletonsMakeItWork}, it follows that $\bigcap_{A\ \text{simple}}M_A\sub\bigcap_{A\in\mathfrak{A}}M_A\sub\V^\B/U$.
\end{proof}

\subsection{\Bukovsky-Dehornoy for ill-founded Boolean ultrapowers}

The proof of Theorem \ref{thm:SimpleSkeletonsMakeItWork} suggests a natural version of the \Bukovsky-Dehornoy phenomenon for ill-founded Boolean ultrapowers. Note that if the models $M_A$ are not well-founded, and hence cannot be taken to be transitive, then it does not make sense to look at $\bigcap_A M_A$.

In order to formulate the version of \Bukovsky-Dehornoy for ill-founded models, we need a little bit of notation. If $M=\kla{|M|,\in^M}$ is a model of set theory, where $\in^M$ is $M$'s interpretation of $\in$, then any member $a$ of $M$ can be viewed as a subset $a^c$ of $|M|$ by setting $a^c=\{b\in|M|\st M\models b\in a\}$. This is often referred to as the set coded by $a$. Let us write
\[M^c=\{a^c\st a\in|M|\}\]
So $M^c$ is the set of sets coded by members of $|M|$. In the context of ill-founded Boolean ultrapowers, a maximal antichain $A\sub\B$ is simple if $G_A:=\pi_A,\infty^{-1}``([\dot{G}]_U)^c$ is coded in $M_A$, and the version of Theorem \ref{thm:ImportanceOfG_A} is as follows.

\begin{thm}
\label{thm:ImportanceOfG_A-Illfounded}
Let $A\sub\B$ be a maximal antichain, $U$ an ultrafilter on $\B$, $j:\V\emb{U}\check{\V}_U$ and $\pi_{0,A}\V\emb{U_A}M_A$. Then
\begin{enumerate}
\item
There is an isomorphism $\chi:\V^\B/U\To M_A^{\pi_{0,A}(\B)}/G_A$.
\item
Letting $j':M_A\emb{G_A}\check{M_A}^{\pi_{0,A}(\B)}_{G_A}$ be the Boolean ultrapower embedding,
$\chi\circ\pi_{A,\infty}=j'$.
\item
\label{item:GenericsAreTheSame-Illfounded}
$\chi\rest([\dot{G}]_U)^c:[\dot{G}]_U\isomorphism\tilde{G}$,
where $\tilde{G}=([\pi_A(\dot{G})]_{G_A})^c$.
\end{enumerate}
\end{thm}

In the context of ill-founded Boolean ultrapowers, the definition of a skeleton has to be modified in much the same way we had to modify the meaning of simplicity of a maximal antichain, so as to require that $j``\mathfrak{A}$ be coded in $\V^\B/U$, rather than just saying that $j``\mathfrak{A}\in\V^\B/U$.

\begin{thm}
\label{thm:Bukovsky-DehornoyIllfounded}
Suppose $U$ is an ultrafilter on $\B$ that may give rise to an ill-founded ultrapower, and suppose there is a skeleton $\mathfrak{A}$ for $(\B,U)$. Then
\[\bigcap_{A\in\mathfrak{A}}\{x\sub\check{\V}_U\st \pi_{A,\infty}^{-1}``x\in (M_A)^c\}\sub\{x\sub\check{\V}_U\st x\in(\V^\B/U)^c\}\]
The converse of this inclusion is equivalent to saying that $\mathfrak{A}$ is simple. Moreover, if every maximal antichain is simple, then
\[\bigcap_{A}\{x\sub\check{\V}_U\st \pi_{A,\infty}^{-1}``x\in (M_A)^c\}=\{x\sub\check{\V}_U\st x\in(\V^\B/U)^c\}\]
\end{thm}

\begin{proof}
Let $a\sub\check{\V}_U$ be such that for every $A\in\mathfrak{A}$, $\pi_{A,\infty}^{-1}``a$ is coded in $M_A$. Let $f_A:A\To\V$ be such that $\pi_{A,\infty}^{-1}``a$ is coded by $[f]_{U_A}$ in $M_A$.

In $\V^\B/U$, if $B$ is a maximal antichain in $j(\B)$ with $B\in\check{\V}_U$, let
$\tilde{b}_B$ be the unique condition in $B\cap G$. Let $j``\mathfrak{A}=D^c$, $D\in\V^\B/U$. Define, in $\V^\B/U$, the set $a'$ by:
\[a'=\{x\st\exists A\in D\forall B\in D (B\le^*A\implies
x\in j(\vf)_B(\tilde{b}_B))\}\]
It follows then as in the proof of \ref{thm:Bukovsky-DehornoyFromSimpleSkeleton} that $a'$ codes $a$.

For the second part of the theorem, first suppose that the reverse inclusion holds. Then, since $([\dot{G}]_U)^c$ is an element of the right hand side, it belongs to the set on the left hand side, which means that for every $A\in\mathfrak{A}$, $\pi_{A,\infty}^{-1}``([\dot{G}]_U)^c$ is coded in $M_A$, that is, $A$ is simple. So $\mathfrak{A}$ is simple.

Conversely, suppose that every $A\in\mathfrak{A}$ is simple. We have to show that the reverse inclusion holds. So suppose $x\sub\check{\V}_U$ is coded in $\V^\B/U$. Given $A\in\mathfrak{A}$, we have to show that $\pi_{A,\infty}^{-1}``x$ is coded in $M_A$. Let $g\in M_A$ code $\pi_{A,\infty}^{-1}``([\dot{G}]_U)^c$. Then we can define within $M_A$ the Boolean ultrapower of $\pi_{0,A}(\B)$ by $g$, and it will be the case that this ultrapower is isomorphic to $(M_A)^{\pi_{0,A}(\B)}/G_A$, so we will just take them to be equal. Let $j'$ be the canonical embedding from $M_A$ into its ultrapower by $G_A$ (i.e., for $a\in M_A$, $j'(a)=b$, where $b$ is the member of $M_A$ such that $M_A$ thinks that $b$ is the image of $a$ under the ultrapower embedding).
Let $\chi:\V^\B/U\To(M_A)^{\pi_{0,A}(\B)}/G_A$ be the isomorphism postulated in Theorem \ref{thm:ImportanceOfG_A-Illfounded}. Then $\pi_{A,\infty}^{-1}``x=(j')^{-1}``\chi``x$, and $\chi``x$ is coded in what $M_A$ thinks is $M_A^{\pi_{0,A}(\B)}/G_A$ (if $x=y^c$, $y\in\V^\B/U$, then $\chi``x=\chi(y)$). Then what $M_A$ thinks is the pullback of the set coded by $y$, codes $\pi_{A,\infty}^{-1}``x$.

The above reasoning works for any simple maximal antichain, so the equality claimed in the theorem follows as well. \end{proof}

\subsection{Characterizing simplicity of antichains}

Fixing a maximal antichain $A\sub\B$, we want to understand better what it means for $A$ to be simple, i.e., for $G_A$ to be in $M_A$. As usual, we assume the Boolean ultrapower to be well-founded. We first analyze what it means for a function $\vec{G}=\kla{G_a\st a\in A}$ to represent $G_A$. That is, we want to characterize when $[\vec{G}]_{U_A}=G_A$. First, recall that
\claim{(1)}{$[\id]_{U_A}\in G_A$.}

\prooff{(1)} This is because $b_A=\pi_{A,\infty}([\id]_{U_A})\in G$. \qedd{(1)}

\medskip

Let us say that a function $f:A\To\B$ such that for all $a\in A$, $f(a)\le a$, is a \emph{pressing down function}, and
let us fix such a function for now. Using the definition of $G_A$, it follows that

\claim{(2)}{$[f]_{U_A}\in G_A\iff\bigvee\{f(a)\st a\in A\}\in U$.}

This is because $[f]_{U_A}\in G_A$ iff $\pi_{A,\infty}([f]_{U_A})\in G$. Now $\pi_{A,\infty}([f]_{U_A})=[\mix\tilde{f}]_U$, where $\mix\tilde{f}$ is the result of mixing the names $(f(a))\check{}$, for $a\in A$, so that $a\le\BV{\mix\tilde{f}=(f(a))\check{}}$. But since $f(a)\le a$ and $A$ is an antichain, it follows that $a=\BV{\mix\tilde{f}=(f(a))\check{}}$. Since furthermore, $G=[\dot{G}]_U$, we get that $[f]_{U_A}\in G_A$ iff $[\mix\tilde{f}]_U\in[\dot{G}]_U$. This latter statement is the case iff $\BV{\mix\tilde{f}\in\dot{G}}\in U$. But
\begin{ea*}
\BV{\mix\tilde{f}\in\dot{G}}&=&\bigvee_{b\in\B}(\BV{\mix\tilde{f}=\check{b}}\land\underbrace{\BV{\check{b}\in\dot{G}}}_{=b})\\
&=&\bigvee_{a\in A}\BV{\mix\tilde{f}=(f(a))\check{}}\land f(a)\\
&=&\bigvee_{a\in A}a\land f(a)\\
&=&\bigvee_{a\in A}f(a).
\end{ea*} \qedd{(2)}

So if $G_A=[\vec{G}]_{U_A}$, it has to be the case for any pressing down function $f:A\To\B$ that
\claim{(3)}{$\bigvee\{f(a)\st a\in A\}\in U\iff\bigvee\{a\in A\st f(a)\in G_a\}\in U$}

This reasoning can be carried out in both directions, and so, we get the following characterization.

\begin{lem}
\label{lem:CharacterizationOfRepresentingG_A}
If $U\sub\B$ is an ultrafilter that gives rise to a well-founded ultrapower, $A\sub\B$ is a maximal antichain and $\vG=\seq{G_a}{a\in A}$ is a function, then the following are equivalent:
\begin{enumerate}
  \item $[\vG]_{U_A}=G_A$
  \item $\{a\in A\st G_a\ \text{is an ultrafilter on}\ \B\}\in U_A$, and for every pressing down function $f:A\To\B$,
      \[\bigvee\{f(a)\st a\in A\}\in U\iff\bigvee\{a\in A\st f(a)\in G_a\}\in U.\]
\end{enumerate}
\end{lem}

It is shown in \cite{Fuchs:StrongPrikry} that if $\P$ is \Prikry{} forcing, Magidor forcing or the generalized \Prikry{} forcing, $\B$ is the Boolean algebra of $\P$, $j:\V\To M$ is the elementary embedding from the imitation iteration of $\P$, and $U$ is the pullback of the ultrafilter on $j(\B)$ generated by the critical sequence, then every
maximal antichain in $\P$ is simple with respect to $U$. It is also show that in the case of \Prikry{} forcing, Magidor forcing and short generalized \Prikry{} forcing, there is a skeleton. It follows that in these cases, the \Bukovsky-Dehornoy phenomenon holds.

\section{Continuous, eventually uniform representations}
\label{sec:CEU}

We would like to find a way to say something about the intersection model without assuming the existence of a skeleton. For the remainder of this section, let us fix a partial order $\P$, its Boolean algebra $\B$, and ultrafilter $U$ on $\B$ such that $\check{\V}_U$ is well-founded, and let $G=[\dot{G}]_U$.

The requirement in Theorem \ref{thm:SimpleSkeletonsMakeItWork} that the set of generating maximal antichains have size at most the critical point of the Boolean ultrapower embedding is very restrictive (for example, it does not apply to medium or long generalized \Prikry{} forcing), but allows us to make the strong conclusion that $\V^\B/U=\bigcap_AM_A$. Asking that the maximal antichains be simple, on the other hand, seems like a very natural and not too restrictive condition.

We will now explore a strengthening of the simplicity of a set, thus arriving at the concept of a continuous, eventually uniform representation. This will enable us to develop another sufficient criterion for when a poset satisfies the \Bukovsky-Dehornoy phenomenon with respect to an ultrafilter on its Boolean algebra. We will also be able to state a general theorem characterizing exactly which part of the intersection model the model $\V^\B/U$ is made up of.

As a motivation for the concepts to follow, suppose $x\sub\check{\V}_U$ is not only in every $M_A$, but assume moreover that it is uniformly represented, in the following sense.

\begin{defn}
\label{def:UniformlyRepresented}
A set $x\sub\check{\V}_U$ is uniformly represented by $\vx=\seq{x_p}{p\in\P}$ with respect to the ultrafilter $U$ on the Boolean algebra $\B$ of $\P$ if for every maximal antichain $A\sub\P$, $x_A=[\vx\rest A]_{U_A}$, where $x_A=\pi_{A,\infty}^{-1}``x$. In this case, $\vx$ is called a uniform representation of $x$ with respect to $U$.
\end{defn}

Let $\vx$ be a uniform representation of $x$ with respect to $U$.
Fix $a\in\check{V}_U$. Let $A\sub\P$ be a maximal antichain such that $a\in\ran(\pi_{A,\infty})$. Then \[\pi_{A,\infty}(x_A)=\pi_{A,\infty}([\vx\rest A]_{U_A})=j(\vx)_{b_A}\]
So $a\in x$ iff $a\in j(\vx)_{b_A}$. The dependency on $b_A$ is what stands in our way and prevents us from turning this into a definition of $x$ in $\V^\B/U$.

In the following, we will mostly focus on the intersection model $\bigcap_{A\sub\P}M_A$, and, accordingly, we will work with uniform representations indexed by members of $\P$. The reader who wants to work with the Boolean algebra directly may just identify $\P$ and $\B$ in what follows. As usual, we assume that $\P$ is separative and has a maximal element, $\eins$.

\begin{defn}
\label{defn:UniformBelowAandEventuallyUniform}
If $A\sub\P$ is a maximal antichain, then let $\P_{\le A}=\{p\st\exists q\in A\quad p\le q\}$. A function $\vx=\seq{x_p}{p\in\P_{\le A}}$ is a uniform representation (wrt.~$U$) of $x$ \emph{below} $A$, where $x\sub\check{\V}_U$, if for every maximal antichain $B\le^* A$, $[\vx\rest B]_{U_B}=x_B$. It is an eventually uniform representation (wrt.~$U$) of $x$ if there is a maximal antichain below which it is a uniform representation.

An eventually uniform representation $\vx$ of a set $x$ is \emph{continuous} if for every $y$, and every $p\in\P$, there is a $q\le p$ such that for all $r_1,r_2\le q$, $y\in x_{r_1}$ iff $y\in x_{r_2}$.

If $x\sub\V^\B/U$ has a continuous, eventually uniform representation, then we just say that $x$ is CEU. Since we fixed $U$ for this section, we may drop the reference to it.
\end{defn}

We will need the following simple observation.

\begin{obs}
\label{obs:b_AsGenerateG}
For any antichain $A\sub\P$ and any $p\in G$, there is an antichain $B\le^* A$ such that $b_B\le p$.
\end{obs}

\begin{proof} By replacing $A$ with a refinement of itself, we may assume that $p$ is in the range of $\pi_{A,\infty}$. Let $\bp$ be the preimage of $p$ under $\pi_{A,\infty}$, hence $\bp\in G_A$. We have seen that $G_A$ is an ultrafilter on $\pi_A(\B)$, and that $[\id]_{U_A}\in G_A$ - see the beginning of this section. This means that there is a $q\in G_A$ with $q\le\bp,[\id]_{U_A}$. Let $q=[f]_{U_A}$, where $f:A\To\P$ is a pressing down function (this can be done as $[f]_{U_A}\le[\id]_{U_A}$). Then $\{f(a)\st a\in A\}$ is an antichain, and we can let $B\supseteq\{f(a)\st a\in A\}$ be a maximal antichain that refines $A$. We claim that $b_B\le p$. To see this, note that $j(f)(b_A)=\pi_{A,\infty}([f]_{U_A})=\pi_{A,\infty}(q)\in G$, because $q\in G_A$. But also, $\ran(j(f))\sub j(B)$, and the latter is a maximal antichain in $j(\P)$. So $q=j(f)(b_A)=b_B$, and $q\le p$. \end{proof}

\begin{remark}
It is not generally the case that every condition $p$ in $G$ is of the form $b_A$, for some antichain $A\sub\P$.
\end{remark}

\begin{proof} For example, in the case of \Prikry{} forcing, the condition $p=\kla{\leer,[\kappa,j(\kappa))}$ is in $G$, but it cannot be of the form $b_A$, where $A\sub\P$ is a maximal antichain, because if $p\in j(A)$, then $p$ is the unique member of $j(A)$ that has empty first coordinate (since $j(A)$ is an antichain in $j(\P)$ and any two conditions with the same first coordinate are compatible). But if $j(A)$ has a unique element with empty first coordinate, then the same must be true of $A$. So then, $p$ would have to be in the range of $j$, but this is impossible, since $\kappa\notin\ran(j)$. \end{proof}

\subsection{Uses and existence of CEU representations}

The following lemma shows that continuous, eventually uniform representations are useful. It shows that enhancing ``$x$ belongs to the intersection model'' by ``$x$ has a continuous, eventually uniform representation'' allows the conclusion that $x$ is in $\V^\B/U$.

\begin{lem}
\label{lem:UseOfContinuousEventuallyUniformRepresentations}
If $x\sub\check{\V}_U$ is CEU, then $x\in\V^\B/U$.
\end{lem}

\begin{proof} Suppose $\vx$ is a continuous uniform representation of $x$ below $A$. We want to use $j(\vx)$ to define $x$ in $\V^\B/U$, as follows:

\[z\in x \iff \exists q\in j(\P)_{\le j(A)} (q\in G\ \text{and for all $r\le q$}, z\in j(\vx)_r)\]

To see that this defines $x$, let $z$ be given.

Since $\vx$ is continuous, $j(\vx)$ is continuous in $\check{V}_U$. So, applying this to $z$, in $\check{\V}_U$, it is true that for every $p\in j(\P)$, there is a $q\le p$ such that for all $r_1,r_2\le q$, $z\in j(\vx)_{r_1}$ iff $z\in j(\vx)_{r_2}$. So the set of such $q$ is dense in $j(\P)$. By genericity, there is such a $q\in G$.

If $z\in x$, then let $B\le^*A$ be such that $z\in\ran(\pi_{B,\infty})$. Let $\pi_{B,\infty}(\bar{z})=z$, so $\bar{z}\in x_B=[\vx\rest B]_{U_B}$. Applying $\pi_{B,\infty}$ on both sides yields that $z\in j(\vx)_{b_B}$. But we can pick $B$ so that $b_B\le q$, by Observation \ref{obs:b_AsGenerateG}. So this means that there is an $r_1\le q$ with $z\in j(\vx)_{r_1}$, and so, this is true for all $r_1\le q$.

On the other hand, assume $z\notin x$. We want to show that in this case, for all $r\le q$, $z\notin j(\vx)_r$. By the choice of $q$, it would otherwise follow that for all $r\le q$, $z\in j(\vx)_r$. But then, we can argue as above: let $B\le^*A$ be a maximal antichain with $b_B\le q$ and $z\in\ran_{B,\infty}$. Then $z\notin j(\vx)_{b_B}$, and $b_B\le q$, a contradiction. So it must be that for all $r\le q$, $z\notin j(\vx)_r$. \end{proof}

The obvious question is now whether it is reasonable to expect continuous, eventually uniform representations to exist. It is shown in \cite{Fuchs:StrongPrikry} that in all the cases considered here, \Prikry{} forcing, Magidor forcing and generalized \Prikry{} forcing of any length, the canonical ultrafilter has a uniform representation, so let's view the assumption that $G$ has a uniform representation as a reasonable one. Suppose $x\in\bigcap_A M_A$. The question is whether it is reasonable to ask that $x$ have a continuous, eventually uniform representation, in order to conclude that $x\in\V^\B/U$, or whether this would be unduly restrictive.
If what we ask of $x$ is implied by $x\in\V^\B/U$, then we are not asking too much. It turns out that it is not asking too much: every element of $\V^\B/U$ that is a subset of $\check{\V}_U$ has a continuous, eventually uniform representation, if $G$ has a uniform representation!

\begin{lem}
\label{lem:StrongPrikryImpliesContinuousRepresentationsForMembersOfV^B/U}
If $G$ is uniformly represented with respect to $U$, and if
$x\in\V^\B/U$ with $x\sub\check{V}_U$, then $x$ is CEU.
\end{lem}

\begin{proof} Fix a uniform representation $\seq{G_p}{p\in\P}$ of $G$ with respect to $U$. Viewing $\P$ as a subset of $\B$,
we may assume that for every $p\in\P$, $G_p$ is an ultrafilter on $\B$ with $p\in G_p$. For if we define $G'_b$ to be equal to $G_p$ if $G_p$ is an ultrafilter on $\B$ with $p\in G_p$, and otherwise we let $G'_p$ be a randomly chosen ultrafilter on $\B$ with $p\in G'_p$, then for every maximal antichain $A\sub\P$, it follows that $[\kla{G_a\st a\in A}]_{U_A}=[\kla{G'_a\st a\in A}]_{U_A}$. This is because in $M_A$, it is true that $[\kla{G_a\st a\in A}]_{U_A}$ is an ultrafilter on $[\const_\B]_{U_A}$ with $[\id]_{U_A}\in[\kla{G_a\st a\in A}]_{U_A}$. This means by \Los{} that the set $X=\{a\in A\st G_a\ \text{is an ultrafilter on $\B$  with}\ a\in G_a\}\in U_A$. For $a\in X$, $G_a=G'_a$, so $[\kla{G_a\st a\in A}]_{U_A}=[\kla{G'_a\st a\in A}]_{U_A}$.

Let $x=\dot{x}^G$, $\dot{x}\in\check{\V}_U$. Let $\pi_{A,\infty}(\dot{\bx})=\dot{x}$, $\dot{\bx}\in M_A$. Let $[f]_{U_A}=\dot{\bx}$. For $p\in\P_{\le A}$, let $h(p)=f(a)^{G_p}$, in the sense that $h(p)=\{y\st\BV{\check{y}\in f(a)}\in G_p\}$, where $a\in A$ is unique with $p\le a$.
We claim that $h$ is a continuous uniform representation of $x$ below $A$. First, let's show that it is a uniform representation of $x$ below $A$. Let $B\le^*A$ be a maximal antichain. We have to show that $[h\rest B]_{U_B}=x_B$. Since $x_B=\pi_{B,\infty}^{-1}``x$, this amounts to showing that for $g:B\To\V$, \[[g]_{U_B}\in[h\rest B]_{U_B}\ \text{iff}\ j(g)(b_B)\in x\]
From left to right, let $[g]_{U_B}\in[h\rest B]_{U_B}$. Note that for $p\in B$, $h(p)=(f^A_B(p))^{G_p}$, where $f^A_B$ is the projection of $f$ onto $B$, so $f^A_B(b)=f(a)$ where $a\in A$ is unique with $b\le a$. So
\[M_B\models [g]_{U_B}\in([f^A_B]_{U_B})^{[\vG\rest B]_{U_B}}\]
In other words,
\[M_B\models [g]_{U_B}\in([f^A_B]_{U_B})^{G_B}\]
So
\[M_B\models \BV{([g]_{U_B})\check{}\in([f^A_B]_{U_B})}\in G_B\]
Now we want to apply $\pi_{B,\infty}$ to the constants in this statement. Note that \[\pi_{B,\infty}([f^A_B]_{U_B})=\pi_{B,\infty}(\pi_{A,B}([f]_{U_A}))=
\pi_{A,\infty}([f]_{U_A})=\pi_{A,\infty}(\dot{\bx})=\dot{x}\]%
So the result is that
\[\BV{(j(g)(b_B))\check{}\in\dot{x}}^{M_B}\in G\]
which means that $j(g)(b_B)\in x$.

For the converse, note that $G_B$ is an ultrafilter on $\pi_B(\B)$. So if $[g]_{U_B}$ is not in $([h\rest B]_{U_B})^{G_B}$, which we understood to mean that $\BV{([g]_{U_B})\check{}\in[h\rest B]_{U_B}}\notin G_B$, then this means that
$\BV{([g]_{U_B})\check{}\notin[h\rest B]_{U_B}}\in G_B$, and applying $\pi_{B,\infty}$ as above, we get that $\BV{(j(g)(b_B))\check{}\notin\dot{x}}\in G$, so that $j(g)(b_B)\notin x$. This proves the desired equivalence, and hence that $h$ is a uniform representation of $x$.

To check that $h$ is continuous, let $y$ be given, and let $p\in\P$. Wlog, let $p\in\P_{\le A}$. Let $a\in A$ be such that $p\le a$. By definition then, $h(p)=
\{z\st\BV{\check{z}\in f(a)}\in G_p\}$. Consider $\BV{\check{y}\in f(a)}$. If this is $0$, then clearly, every extension $r$ of $p$ will have $y\notin h(r)$. More generally if $\BV{\check{y}\in f(a)}$ is incompatible with $p$, then the same will be true for every extension $r$ of $p$, because every such $G_r$ will have $p\in G_r$, and so, $\BV{\check{y}\in f(a)}\notin G_r$. But if $\BV{\check{y}\in f(a)}$ is compatible with $p$, then we can pick $q\le p, \BV{\check{y}\in f(a)}$, and as a result, for every $r\le q$, $\BV{\check{y}\in f(a)}\in G_r$, which implies that $y\in h(r)$.

So $h$ is a continuous uniform representation of $x$ below $A$. \end{proof}

In particular, if $G$ is uniformly represented, then it is continuously uniformly represented. Indeed, if $\vec{G}$ is a uniform representation of $G$ such that for every $p\in\P$, $G_p$ is an ultrafilter on $\B$ with $p\in G_p$, then that representation is continuous. In \cite{Fuchs:StrongPrikry}, a sufficient criterion called the strong \Prikry{} property is given, that is in some sense equivalent to $G$ being uniformly represented.

\begin{thm}
\label{thm:V^B/UTheSameAsTheSetsWithContinuousEvUniformRepresentations}
Suppose $G$ is uniformly represented with respect to $U$.
Then for $x\sub\check{\V}_U$, the following are equivalent:
\begin{enumerate}
  \item $x$ is CEU wrt.~$U$.
  \item $x\in\V^\B/U$.
\end{enumerate}
\end{thm}

\begin{proof} The implication from 2.~to 1.~is Lemma \ref{lem:StrongPrikryImpliesContinuousRepresentationsForMembersOfV^B/U}, and the implication from 1.~to 2.~is Lemma \ref{lem:UseOfContinuousEventuallyUniformRepresentations}.
\end{proof}

\subsection{Conclusions about \Bukovsky-Dehornoy, and an alternative intersection model}

Concerning the intersection model, we get:

\begin{thm}
If $G$ is uniformly represented with respect to $U$, then the following are equivalent:
\begin{enumerate}
  \item
  \label{item:B-Dholds}
  $\bigcap_AM_A=\V^\B/U$
  \item
  \label{item:ThingsInTheIntersectionModelHaveContinuousUniformReps}
  Every $x\in\bigcap_{A\sub\P}M_A$ with $x\sub\check{V}_U$ is CEU wrt.~$U$.
\end{enumerate}
\end{thm}

\begin{proof} For \ref{item:B-Dholds}$\implies$\ref{item:ThingsInTheIntersectionModelHaveContinuousUniformReps}, let $x\in\bigcap_AM_A$, $x\sub\check{V}_U$. By \ref{item:B-Dholds}, $x\in\V^\B/U$. By Lemma \ref{lem:StrongPrikryImpliesContinuousRepresentationsForMembersOfV^B/U}, $x$ has a continuous, eventually uniform representation.

For the converse direction, note that since $G$ has a uniform representation, every maximal antichain is simple, and hence, $\V^\B/U\sub\bigcap_A M_A$, by Lemma \ref{lem:EasyDirectionForSimpleAntichainsGlobal}. It suffices to prove the reverse inclusion for sets of ordinals, as in the proof of Theorem \ref{thm:SimpleSkeletonsMakeItWork}. But if $a\in\bigcap_A M_A$ is a set of ordinals, then $a\sub\check{\V}_U$, and so, by \ref{item:ThingsInTheIntersectionModelHaveContinuousUniformReps}, $a$ has a continuous, eventually uniform representation, so by Lemma \ref{lem:UseOfContinuousEventuallyUniformRepresentations}, $a\in\V^\B/U$.
\end{proof}

For the purpose of the following theorem, let us say that a binary relation $a\sub\On\times\On$ is \emph{a code for the set $x$} if, letting $b$ be the field of the $a$ (i.e., the set of ordinals that occur as first or second coordinates of elements of $a$), $\kla{b,a}$ is extensional and well founded, and the Mostowski-collapse of $\kla{b,a}$ is $\kla{\TC(\{x\}),\in}$.

This concept allows us to characterize $\V^\B/U$ as a modified intersection model. The result may be viewed as a generalization of the original \Bukovsky-Dehornoy phenomenon.

\begin{thm}
\label{thm:TheCorrectIntersectionModel}
If $G$ is uniformly represented with respect to $U$, then
\[\V^\B/U=\{x\in\bigcap_{A\sub\P} M_A\st x\ \text{has a CEU code wrt.~ $U$}\}\]
\end{thm}

\begin{proof} For the inclusion from left to right, suppose $x\in\V^\B/U$. Since $\V^\B/U$ is a model of $\ZFC$, there is a code $a$ for $x$ in $\V^\B/U$. By Lemma \ref{lem:StrongPrikryImpliesContinuousRepresentationsForMembersOfV^B/U}, $a$ is CEU.
Moreover, since $G$ is uniformly represented, every maximal antichain is simple, and so, it follows that $\V^\B/U\sub\bigcap_A M_A$, by Lemma \ref{lem:EasyDirectionForSimpleAntichainsGlobal}, as before. So $x\in M_A$, for every maximal antichain $A$.

For the other direction, let $a$ be a CEU code for $x$, where $x\in\bigcap_AM_A$.
By Theorem \ref{thm:V^B/UTheSameAsTheSetsWithContinuousEvUniformRepresentations}, $a\in\V^\B/U$. But then, $a$ can be decoded in $\V^\B/U$, so that $x\in\V^\B/U$. \end{proof}


\end{document}